\documentclass[11pt,reqno]{amsart}

\usepackage{amsfonts, amsthm, amsmath}
\allowdisplaybreaks[4]

\usepackage{rotating}

\usepackage{tikz}

\usepackage{graphics}

\usepackage{amssymb}

\usepackage{amscd}

\usepackage[latin2]{inputenc}

\usepackage{t1enc}

\usepackage[mathscr]{eucal}

\usepackage{indentfirst}

\usepackage{nicematrix}

\usepackage{graphicx}

\usepackage{graphics}

\usepackage{pict2e}

\usepackage{mathrsfs}

\usepackage{enumerate}
\usepackage[pagebackref]{hyperref}
\hypersetup{colorlinks=true}
\usepackage{color}
\usepackage{epic}
\usepackage{bm}
\usepackage{hyperref} %this gives clickable references, which is nice
\usepackage{framed}
\usepackage{mathabx}
\usepackage{stmaryrd}

\numberwithin{equation}{section}
\theoremstyle{plain}

\newtheorem{theorem}{Theorem}[section]

\newtheorem{lemma}[theorem]{Lemma}

\theoremstyle{definition}

\newtheorem{Def}[theorem]{Definition}

\newtheorem{example}[theorem]{Example}

\newtheorem{remark}[theorem]{Remark}

\newtheorem{?}[theorem]{Problem}

\def\la{\lambda}
\def\sR{\mathscr{R}}
\def\sP{\mathscr{P}}
\def\sI{\mathscr{I}}
\def\cL{\mathcal{L}}
\def\ep{\epsilon}
\def\diag{\mathrm{diag}}
\def\rA{\mathrm{A}}
\def\sA{\mathscr{A}}
\def\sB{\mathscr{B}}
\def\rW{\mathrm{W}}
\def\uA{\underline{\mathbf{A}}}
\def\ual{\underline{{\bm\alpha}}}
\def\ubeta{\underline{{\bm\beta}}}
\def\uga{\underline{{\bm\gamma}}}
\def\ugaI{\underline{{\bm\gamma}}_1}
\def\ugaII{\underline{{\bm\gamma}}_2}

\def\ri{\rightarrow}
\def\al{\alpha}
\def\mat{\mathrm{Mat}}
\def\bN{\mathbb{N}}
\def\ga{\gamma}
\def\bZ{\mathbb{Z}}
\def\sS{\mathscr{S}}
\def\sE{\mathscr{E}}
\def\mod{\mathrm{mod}}
\def\bC{\mathbb{C}}
\def\bn{\mathbf{n}}

\begin{document}

\title[Linked partition ideals and gap-frequency partitions]{Linked partition ideals and gap-frequency partitions}

\author[H. Li]{Haijun Li}
\address[Haijun Li]{College of Mathematics and Statistics, Chongqing University, Chongqing 401331, P.R. China}
\email{lihaijun@cqu.edu.cn; lihaijune@163.com}

\date{\today}

\begin{abstract}
Recently, linked partition ideals have attracted renewed attention. In this paper, we embed the gap-frequency partitions in which each part appears at most twice or three times into the framework of span one linked partition ideals, and derive refined generating functions for both cases. Furthermore, for gap-frequency partitions in which every part appears at most $k$ times, we provide a purely combinatorial proof of the corresponding refined generating function.
\end{abstract}


% Secondary 05A15, 05A19, 05E18, 11A55, 11B68.
\maketitle
%%%%%%%%%%%%%%%%%%%%%%%%%%%%%%%%%%%%%%%%%%%%%%%%%%%%%%%%
\section{Introduction}

For a nonnegative integer $n$, a partition $\la$ of $n$ is a finite non-decreasing sequence of positive integers $\la=\la_1+\la_2+\cdots+\la_{\ell}$ whose sum is $n$. Usually, the {\it weight} of $\la$ is denoted by $|\la|=n$, each $\la_i$ is called the {\it part} of $\la$ and  the number of parts $\ell$ is called the {\it length} of $\la$, denoted by $\ell(\la)$. 

One of the central concerns in partition theory is the following Rogers-Ramanujan identities \cite{ram14, ram19, rog94}, stated here in the partition version due to MacMahon \cite{mac16} and Schur \cite{sch17}.

\begin{theorem}[Rogers-Ramanujan identities, partition version]\label{thm:RRid}
For any integer $n\geq 1$,
\begin{itemize}
\item[(1)] the number of partitions of $n$ into parts congruent to $\pm 1$ modulo $5$ is the same as the number of partitions of $n$ such that every two consecutive parts have difference at least $2$;

\item[(2)] the number of partitions of $n$ into parts congruent to $\pm 2$ modulo $5$ is the same as the number of partitions of $n$ such that every two consecutive parts have difference at least $2$ and that the smallest part is greater than $1$.
\end{itemize}
\end{theorem}

There have been numerous generalizations and analogs of the Rogers-Ramanujan identities, among which Gordon's extension \cite{gor61} to higher moduli is of substantial significance.

\begin{theorem}[Gordon's identities]\label{thm:Gid}
Let $k$ and $i$ be positive integers such that $k\geq 2$ and $1\leq i\leq k$. Let $\sA_{k, i}$ be the set of partitions $\la$ where $\la_{j+k-1}-\la_j\geq 2$ and at most $(i-1)$ of the parts are equal to $1$. Let $\sB_{k, i}$ be the set of partitions into parts that are not congruent to $0$, $\pm i$ modulo $(2k+1)$. For an integer $n\geq 0$, let $A_{k, i}(n)$ (resp. $B_{k, i}$) denote the number of partitions of $n$ which belong to $\sA_{k, i}$ (resp. $\sB_{k, i}$). Then we have
\begin{align*}
A_{k, i}(n)=B_{k, i}(n).
\end{align*}
\end{theorem}

It is clear that the Rogers-Ramanujan identities corresponds to the cases $k=i=2$ and $k=i+1=2$ in Theorem \ref{thm:Gid}. Subsequently, Andrews \cite{and742} established the analytic counterpart of Gordon's result.
\begin{theorem}[Andrews-Gordon identities]\label{thm:AGid}
For $k\geq 2$ and $1\leq i\leq k$, we have
\begin{align}
\sum_{n_1, ..., n_{k-1}\geq 0}\frac{q^{N_1^2+N_2^2+\cdots +N_{k-1}^2+N_i+N_{i+1}+\cdots +N_{k-1}}}{(q; q)_{n_1}(q; q)_{n_2}\cdots (q; q)_{n_{k-1}}}=\prod_{\substack{n\geq 1\\ n\not\equiv 0, \pm i\ (\mod\ 2k+1)}}\frac{1}{1-q^n},
\end{align}
where $N_j=n_j+n_{j+1}+\cdots +n_{k-1}$.
\end{theorem}
Here and throughout we adopt the {\it $q$-Pochhammer symbols} \cite{GR04} for $n\in \bN\cup\{\infty\}$, 
\begin{align*}
(A; q)_n:=\prod_{k=0}^{n-1}(1-Aq^k)\ \text{ and }\ (A_1, ..., A_m; q)_n=(A_1; q)_n\cdots (A_m; q)_m\text{ for }m\geq 1.
\end{align*}

In addition to the celebrated Rogers-Ramanujan identities, Rogers also discovered many other interesting sum-product identities. Based on another set of identities about the set $\sB_{3, i}$ in \cite{rog94}, Andrews \cite{and80} investigated a class of partitions subject to difference conditions. A partition $\la$ is said to be a {\it gap-frequency partition} if whenever a part $a$ appears exactly $i$ times then the next strictly larger part is at least $a+i$ and if it is exactly $a+i$ then it can appear at most $i$ times. Thus $1+4+4+4+8+8+8$ is a gap-frequency partition, while $2+5+5+5+5+8+8+8$ is not since $5$ appears four times and $8-5=3<4$. Let $\sS^k$ be the set of gap-frequency partitions where each part appears at most $k$ times. Andrews proved the following generating function for $\sS^k$:
\begin{align}\label{id:g-f gf}
\sum_{\la\in\sS^k}x^{\ell(\la)}q^{|\la|}=\sum_{n_1, ..., n_k\geq 0}&\frac{x^{\sum_{i=1}^k in_i}q^{\sum_{i=1}^k(i^2\binom{n_i}{2}+in_i)+\sum_{1\leq i<j\leq k}ijn_in_j}}{(q; q)_{n_1}(q^2; q^2)_{n_2}\cdots (q^k; q^k)_{n_k}}.
\end{align}

Recently, K\i l\i\c{c} \cite{kil25} gave a new combinatorial interpretation of $\sS^2$ and its companions based on the base-increments framework. For the origins and development of this framework may be found in \cite{kur10, kur19, kur21, KS22}. The first objective of this paper is to embed $\sS^2$ and $\sS^3$ into the framework of span one linked partition ideals, that is, the generating functions of partitions in these sets satisfy recurrence or functional equations similar to those arising from a span one linked partition ideal. We shall not only derive the refined generating functions for partitions belonging to $\sS^2$ and $\sS^3$ but also obtain additional generating functions for such partitions with further restrictions on the smallest part.  Given a partition $\la\in \sP$, we denote by $\ell_k(\la)$ the number of parts of $\la$ that are repeated exactly $k$ times. We will obtain the following refinements.
\begin{theorem}\label{thm:Ak=2}
Let $\sS^2_{i}$ be the set of partitions in $\sS^2$ where at most $i$ of the parts are equal to $1$ for $i=0, 1, 2$ (note that $\sS^2_2=\sS^2$). Then we have
\begin{align}
\sum_{\la\in\sS^2}x^{\ell(\la)}y^{\ell_2(\la)}q^{|\la|}&=\sum_{n_1, n_2\geq 0}\frac{x^{n_1+2n_2}y^{n_2}q^{\binom{n_1}{2}+4\binom{n_2}{4}+2n_1n_2+n_1+2n_2}}{(q; q)_{n_1}(q^2; q^2)_{n_2}},\\
\sum_{\la\in\sS^2_1}x^{\ell(\la)}y^{\ell_2(\la)}q^{|\la|}&=\sum_{n_1, n_2\geq 0}\frac{x^{n_1+2n_2}y^{n_2}q^{\binom{n_1}{2}+4\binom{n_2}{4}+2n_1n_2+n_1+4n_2}}{(q; q)_{n_1}(q^2; q^2)_{n_2}},\\
\sum_{\la\in\sS^2_0}x^{\ell(\la)}y^{\ell_2(\la)}q^{|\la|}&=\sum_{n_1, n_2\geq 0}\frac{x^{n_1+2n_2}y^{n_2}q^{\binom{n_1}{2}+4\binom{n_2}{4}+2n_1n_2+2n_1+4n_2}}{(q; q)_{n_1}(q^2; q^2)_{n_2}}.
\end{align}
\end{theorem}

\begin{theorem}\label{thm:Ak=3}
Let $\sS^3_{i, j}$ be the set of partitions in $\sS^3$ where at most $i$ of the parts are equal to $1$ and at most $j$ of the parts are equal to $2$ for $0\leq i, j\leq 3$ (note that $\sS^3_{3, 3}=\sS^3$). Then we have
\begin{align}
\sum_{\la\in\sS^3}x^{\ell(\la)}y^{\ell_2(\la)}z^{\ell_3(\la)}q^{|\la|}=\sum_{n_1, n_2, n_3\geq 0}&\frac{x^{n_1+2n_2+3n_3}y^{n_2}z^{n_3}}{(q; q)_{n_1}(q^2; q^2)_{n_2}(q^3; q^3)_{n_3}}\nonumber\\
&\times q^{\binom{n_1}{2}+4\binom{n_2}{2}+9\binom{n_3}{2}+2n_1n_2+3n_1n_3+6n_2n_3+n_1+2n_2+3n_3},\\
\sum_{\la\in\sS^3_{2, 3}}x^{\ell(\la)}y^{\ell_2(\la)}z^{\ell_3(\la)}q^{|\la|}=\sum_{n_1, n_2, n_3\geq 0}&\frac{x^{n_1+2n_2+3n_3}y^{n_2}z^{n_3}}{(q; q)_{n_1}(q^2; q^2)_{n_2}(q^3; q^3)_{n_3}}\nonumber\\
&\times q^{\binom{n_1}{2}+4\binom{n_2}{2}+9\binom{n_3}{2}+2n_1n_2+3n_1n_3+6n_2n_3+n_1+2n_2+6n_3},\\
\sum_{\la\in\sS^3_{1, 3}}x^{\ell(\la)}y^{\ell_2(\la)}z^{\ell_3(\la)}q^{|\la|}=\sum_{n_1, n_2, n_3\geq 0}&\frac{x^{n_1+2n_2+3n_3}y^{n_2}z^{n_3}}{(q; q)_{n_1}(q^2; q^2)_{n_2}(q^3; q^3)_{n_3}}\nonumber\\
&\times q^{\binom{n_1}{2}+4\binom{n_2}{2}+9\binom{n_3}{2}+2n_1n_2+3n_1n_3+6n_2n_3+n_1+4n_2+6n_3},\\
\sum_{\la\in\sS^3_{0, 3}}x^{\ell(\la)}y^{\ell_2(\la)}z^{\ell_3(\la)}q^{|\la|}=\sum_{n_1, n_2, n_3\geq 0}&\frac{x^{n_1+2n_2+3n_3}y^{n_2}z^{n_3}}{(q; q)_{n_1}(q^2; q^2)_{n_2}(q^3; q^3)_{n_3}}\nonumber\\
&\times q^{\binom{n_1}{2}+4\binom{n_2}{2}+9\binom{n_3}{2}+2n_1n_2+3n_1n_3+6n_2n_3+2n_1+4n_2+6n_3},\\
\sum_{\la\in\sS^3_{0, 2}}x^{\ell(\la)}y^{\ell_2(\la)}z^{\ell_3(\la)}q^{|\la|}=\sum_{n_1, n_2, n_3\geq 0}&\frac{x^{n_1+2n_2+3n_3}y^{n_2}z^{n_3}}{(q; q)_{n_1}(q^2; q^2)_{n_2}(q^3; q^3)_{n_3}}\nonumber\\
&\times q^{\binom{n_1}{2}+4\binom{n_2}{2}+9\binom{n_3}{2}+2n_1n_2+3n_1n_3+6n_2n_3+2n_1+4n_2+9n_3},\\
\sum_{\la\in\sS^3_{0, 0}}x^{\ell(\la)}y^{\ell_2(\la)}z^{\ell_3(\la)}q^{|\la|}=\sum_{n_1, n_2, n_3\geq 0}&\frac{x^{n_1+2n_2+3n_3}y^{n_2}z^{n_3}}{(q; q)_{n_1}(q^2; q^2)_{n_2}(q^3; q^3)_{n_3}}\nonumber\\
&\times q^{\binom{n_1}{2}+4\binom{n_2}{2}+9\binom{n_3}{2}+2n_1n_2+3n_1n_3+6n_2n_3+3n_1+6n_2+9n_3}.
\end{align}
\end{theorem}

Lastly, we turn our attention to the combinatorics of partition set $\sS^k$. Utilizing a more general base-increments framework (for this framework, see also \cite{FL241}), we present a bijective proof of its generating function, which further preserves a family of additional partition statistics.

\begin{theorem}\label{thm:Ak}
For an integer $k\geq 2$, we have
\begin{align}
\sum_{\la\in\sS^k}x_1^{\ell(\la)}(\prod_{i=2}^kx_{i}^{\ell_i(\la)})q^{|\la|}=\sum_{n_1, ..., n_k\geq 0}&\frac{x_1^{\sum_{i=1}^k in_i}(\prod_{i=2}^kx_i^{n_i})q^{\sum_{i=1}^k(i^2\binom{n_i}{2}+in_i)+\sum_{1\leq i<j\leq k}ijn_in_j}}{(q; q)_{n_1}(q^2; q^2)_{n_2}\cdots (q^k; q^k)_{n_k}}.\label{id:Ak}
\end{align}
\end{theorem}

This paper is organized as follows. Section \ref{sec:pre} reviews the framework of span one linked partition ideals and presents an equivalent definition of $\sS^k$. In Section \ref{sec:set-23}, we study the partition sets $\sS^2$ and $\sS^3$, and prove Theorems \ref{thm:Ak=2} and \ref{thm:Ak=3}. Next, a purely bijective proof of Theorem \ref{thm:Ak} is given in Section \ref{sec:com-k}. Finally, we conclude with some remarks in Section \ref{sec:con}.

%%%%%%%%%%%%%%%%%%%%%%%%%%%%%%%%%%%
\section{Preliminaries}\label{sec:pre}

In this section, we begin by introducing span one linked partition ideals and establishing a key recurrence relation, followed by an equivalent definition of $\sS^k$.

%%%%%%%%%%%%%%%%%%%%%%%%
\subsection{Span one linked partition ideals and a key recurrence relation} The notion of linked partition ideals was initially introduced by Andrews \cite{and741}; see \cite[Chapter 8]{andtp} for details. Much later, Chern \cite{CL20} revisited this framework and significantly extended it, yielding partition-theoretic interpretations for several novel $q$-series identities. Further developments along these lines have been pursued in \cite{AC23, ACL22, che20, che221, che222, che23, GY251, GY252}. We next provide a brief review of span one linked partition ideals.

\begin{Def}[{\cite[Definition 2.1]{ACL22}}]\label{def:LPI}
    Assume that we are given 
    \begin{itemize}
        \item[(1)] a finite set $\Pi=\{\pi_1, \pi_2, ..., \pi_K\}$ of integer partitions with $\pi_1=\ep$, the empty partition,

        \item[(2)] a {\it map of linking sets}, $\cL:\Pi\ri P(\Pi)$, the power set of $\Pi$, with especially, $\cL(\pi_1)=\cL(\ep)=\Pi$ and $\pi_1=\ep\in\cL(\pi_k)$ for any $1\leq k\leq K$,

        \item[(3)] and a positive integer $T$, called the {\it modulus}, which is greater than or equal to the largest part among all partitions in $\Pi$.
    \end{itemize}
    We say a {\it span one linked partition ideal} $\sI=\sI(\langle \Pi, \cL\rangle, T)$ is the collection of all partitions of the form
    \begin{align}
        \la&=\phi^0(\la_0)\oplus\phi^T(\la_1)\oplus\cdots\oplus\phi^{NT}(\la_N)\oplus\phi^{(N+1)T}(\pi_1)\oplus\phi^{(N+2)T}(\pi_1)\oplus\cdots\nonumber\\
        &=\phi^0(\la_0)\oplus\phi^T(\la_1)\oplus\cdots\oplus\phi^{NT}(\la_N),\label{id:decom}
    \end{align}
    where $\la_i\in\cL(\la_{i-1})$ for each $i$ and $\la_N$ is not the empty partition. We also include in $\sI$ the empty partition, which corresponds to $\phi^0(\pi_1)\oplus\phi^T(\pi_1)\oplus\cdots$. Here for any two partitions $\mu$ and $\nu$, $\mu\oplus \nu$ gives a partition by collecting all parts in $\mu$ and $\nu$, and $\phi^m(\mu)$ gives a partition by adding $m$ to each part of $\mu$. Further, $\phi^m(\ep)=\ep$ for any $m$.
\end{Def}

\begin{example}
Let $\sR$ be the set of Rogers-Ramanujan partitions, that is, partitions such that every two consecutive parts have difference at least $2$. Then we know that $\sR=\sI(\langle \Pi_{\sR}, \cL_{\sR}\rangle, 2)$ where $\Pi_{\sR}=\{\ep, 1, 2\}$ and the linking sets are
\begin{align*}
\cL_{\sR}(\ep)=\{\ep, 1, 2\},\quad \cL_{\sR}(1)=\{\ep, 1, 2\},\quad \cL_{\sR}(2)=\{\ep, 2\}.
\end{align*}
This follows from the fact that we can decompose any Rogers-Ramanujan partition as in \eqref{id:decom}. For example, we have
\begin{align*}
\la&=1+3+6+10+15\\
&=\phi^0(1)\oplus\phi^2(1)\oplus\phi^4(2)\oplus\phi^6(\ep)\oplus\phi^8(2)\oplus\phi^{10}(\ep)\oplus\phi^{12}(\ep)\oplus\phi^{14}(1).
\end{align*}
Note that $1\notin \cL_{\sR}(2)$. This is because if $1\in \cL_{\sR}(2)$, then for some $i$ we would have $\phi^{2i}(2)\oplus\phi^{2i+2}(1)$ destroying the gap condition in Rogers-Ramanujan partitions.
\end{example}

On the other hand, Chern \cite{che20} introduced a key recurrence relation for a family of $q$-multi-summations, which he later refined in \cite{che221}. Let $R$ be a fixed positive integer. We fix a symmetric matrix $\ual=(\alpha_{i, j})\in \mat_{R\times R}(\bN)$ and a vector $\uA=(A_r)\in\bN^R_{>0}$. We also fix $J$ vectors $\uga_j=(\ga_{j, r})\in\bN^R_{\geq 0}$ for $j=1, 2, ..., J$. Let $x_1, x_2, ..., x_J$ and $q$ be indeterminates such that the following $q$-multi-summation $H(\ubeta)=H(\beta_1, \beta_2, ..., \beta_R)$ for $\ubeta\in\bZ^R$ converges:
\begin{align}
    H(\ubeta):=\sum_{n_1, ..., n_R\geq 0}\frac{(\prod_{j=1}^Jx_j^{\sum_{r=1}^R\ga_{j, r}n_r})q^{\sum_{r=1}^R\al_{r, r}\binom{n_r}{2}+\sum_{1\leq i<j\leq R}\al_{i, j}n_in_j+\sum_{r=1}^R\beta_rn_r}}{(q^{A_1}; q^{A_1})_{n_1}\cdots (q^{A_R}; q^{A_R})_{n_R}}.\label{id:q-multi}
 %x_1^{\sum_{r=1}^R\ga_{1, r}n_r}\cdots x_J^{\sum_{r=1}^R\ga_{J, r}n_r}
\end{align}
Note that if $F(x_1)=H(\beta_1, ..., \beta_R)$, then $F(x_1q^T)=H(\beta_1+\ga_{1, 1}T, ..., \beta_R+\ga_{1, R}T)$.
\begin{lemma}[{\cite[Lemma 2.1]{che221}}]\label{lem:key rec}
    For $1\leq r\leq R$, we have
    \begin{align}
        H(\beta_1, ..., \beta_r, ..., \beta_R)=&H(\beta_1, ..., \beta_r+A_r, ..., \beta_R)\nonumber\\
        &+x_1^{\ga_{1, r}}\cdots x_J^{\ga_{J, r}}q^{\beta_r}H(\beta_1+\al_{r, 1}, ..., \beta_r+\al_{r, r}, ..., \beta_R+\al_{r, R}).\label{id:key rec}
    \end{align}
\end{lemma}
As described in \cite{che221}, the recurrence relation \eqref{id:key rec} can be illustrated by a binary tree as shown in Figure \ref{fig:rec}, where the coordinate on which the recurrence acts is denoted by $\check{\beta}_r$.

\begin{figure}[htbp]
\begin{center}
\begin{tikzpicture}[
  level 1/.style={level distance=1.5cm, sibling distance=8cm}, % 调整子节点间距和层级距离
  every node/.style={align=center}, % 自动居中对齐，无需parbox的centering
  edge from parent/.style={thick, draw} % 统一边的样式
]
  \node (root) {\parbox{4cm}{\(H(\beta_1, \dots, \check{\beta}_r, \dots, \beta_R)\)}}
    child {
      node {\parbox{5cm}{\(H(\beta_1, \dots, \beta_r + A_r, \dots, \beta_R)\)}}
      edge from parent node[left, midway] {1} % 使用midway确保标签在中间
    }
    child {
      node {\parbox{7cm}{\(H(\beta_1 + \alpha_{r,1}, \dots, \beta_r + \alpha_{r,r}, \dots, \beta_R + \alpha_{r,R})\)}}
      edge from parent node[right, midway] {\(x_1^{\gamma_{1, r}} \cdots x_J^{\gamma_{J, r}} q^{\beta_r}\)}
    };
\end{tikzpicture}
\end{center}
\caption{Node $H(\beta_1, ..., \check{\beta}_r, ..., \beta_R)$ and its children.}
\label{fig:rec}
\end{figure}
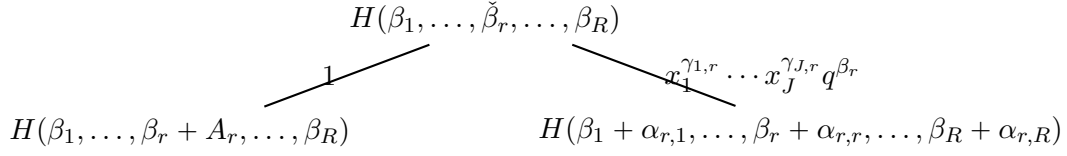

%%%%%%%%%%%%%%%%%%%%%%%%%%%%%%
\subsection{An equivalent definition of $\sS^k$} Given a partition $\la\in \sP$ and a positive integer $a$, if exactly $i$ parts of $\la$ are equal to $a$, then we can define the $i$-tuple $$[a]_i=\underbrace{a+a+\cdots +a}_{i}.$$ Thus, every partition can be expressed as a union of such $i$-tuples, where $i$ is the multiplicity of the corresponding part. Recall that $\sS^k$ denotes the set of gap-frequency partitions where each part appears at most $k$ times. Next, we consider the sets $\sS^1$ and $\sS^2$.

\begin{enumerate}[(1)]
\item For $\sS^1$, it is readily seen to be precisely the set of partitions into distinct parts. Equivalently, every $\la\in \sS^1$ is composed of only $1$-tuples, and for any two consecutive tuples $[a]_1$ and $[b]_1$ appearing in $\la$ with $a<b$, the difference condition $b-a\geq 1$ is automatically satisfied.

%For any $\la\in\sS^1$, we know that $\la$ consists of some $1$-tuples. Then $\sS^1$ is a set of partitions where the difference condition $b-a\geq 1$ must be satisfied for any two consecutive distinct $1$-tuples $[a]_1$ and $[b]_1$.

\item For $\sS^2$, it is in fact the set $\sA_{3, 3}$ appearing in Theorem \ref{thm:Gid}. Hence, every $\la\in\sS^2$ is composed of only $1$-tuples and $2$-tuples. Let $[a]_i$ and $[b]_j$ be two consecutive tuples in $\la$ with $a<b$ where $i, j\in \{1, 2\}$. The required difference $(b-a)$ is summarised in the following matrix:
\begin{align*}
\begin{pNiceMatrix}[first-row, first-col, nullify-dots, columns-width=auto]
 & [b]_1 & [b]_2\\
 [a]_1 & 1 & 2\\
 [a]_2 & 2 & 2
\end{pNiceMatrix},
\end{align*}
each entry of which gives the minimal allowed difference between $a$ and $b$.
\end{enumerate}

Guided by the foregoing examples, we proceed to present an equivalent characterization of $\sS^k$ via explicit difference conditions. Recall that  partition $\la$ is called a {\it gap-frequency partition} if, whenever a part $a$ appears exactly $i$ times then the next strictly larger part $b$ is required to exceed $a$ by at least $i$; moreover, should the difference equal $i$, the number of times $b$ appears must not surpass $i$. In other word, if the number of occurrences of $b$ is less than or equal to $i$, then $b-a\geq i$; if the number of occurrences of $b$ exceeds $i$, then $b-a>i$. This motivates the following formal definition.

%is at least $a+i$ and if it is exactly $a+i$ then it can appear at most $i$ times. We assume that the next larger part after $a$ is $b$, then this condition will become that $b-a\geq i$ if the number of appearance of $b$ is less than or equal to $i$ and $b-a$ is more than $i$ if the number of appearance of $b$ is more than $i$. Therefore, we have the following definition.

\begin{Def}\label{def:sk}
For any integer $k\geq 1$,
% let $[a]_i$ and $[b]_j$ be two consecutive tuples in $\la\in \sS^k$ with $a<b$ where $i, j\in \{1, 2, ..., k\}$.
$\sS^k$ is the set of partitions $\la$ satisfying
\begin{enumerate}[(1)]
\item each part appears at most $k$ times;

\item the required gap between consecutive tuples is summarised in the matrix below:
\begin{align}
\begin{pNiceMatrix}[first-row, first-col, nullify-dots, columns-width=auto]
 & [b]_1 & [b]_2 & [b]_3 & \cdots & [b]_i & [b]_{i+1} & \cdots & [b]_k\\
 [a]_1 & 1 & 2 & 2 & \cdots & 2 & 2 & \cdots & 2\\
 [a]_2 & 2 & 2 & 3 & \cdots & 3 & 3 & \cdots & 3\\
 [a]_3 & 3 & 3 & 3 & \cdots & 4 & 4 & \cdots & 4\\
 \vdots\ \ \ & \vdots & \vdots & \vdots & &\vdots &\vdots & & \vdots\\
 [a]_i & i & i & i & \cdots & i & i+1 & \cdots & i+1\\
 [a]_{i+1} & i+1 & i+1 & i+1 & \cdots & i+1 & i+1 & \cdots & i+2\\
 \vdots\ \ \ & \vdots & \vdots & \vdots & &\vdots &\vdots & & \vdots\\
 [a]_k & k & k & k & \cdots & k & k & \cdots & k
\end{pNiceMatrix},\label{mat:gap}
\end{align}
where the $(i, j)$-entry gives the minimal allowed value of $(b-a)$ for consecutive tuples $[a]_i$ and $[b]_j$ with $a<b$; equivalently,
\begin{align}
b-a\begin{cases}
\geq i, & \text{ if }j\leq i,\\
\geq i+1, & \text{ if }j>i.
\end{cases}
\end{align}
%each entry gives the minimal allowed difference between $a$ and $b$.
\end{enumerate}
\end{Def}

%%%%%%%%%%%%%%%%%%%%%%%%%%%%%%%%%%%%%%%%%%%%%%%%%%%%%%%%
\section{Partition sets $\sS^2$ and $\sS^3$}\label{sec:set-23}

In this section, the partition sets $\sS^2$ and $\sS^3$ are embedded into the framework of span one linked partition ideals. Leveraging the key recurrence relation \eqref{id:key rec}, we then establish Theorems \ref{thm:Ak=2} and \ref{thm:Ak=3}.

%Now we first consider the following partition set where $k=3$. Let $\sR_{3,a}$ be the set of partitions $\la$ where $\la_{i+2}-\la_i\geq 2$ and the number of $1$'s is at most $(a-1)$. And denote by $\sP_{S_m}$ the set of partitions in $\sP$ whose $m$-tail is not in the set $S_m$ of partitions. Note that these identities are due to Andrews. Let $\sE_k$ denote the set of partitions into multiples of $k$ and set $\sE_{k,n}:=\{\la\in \sE_k: \ell(\la)\leq n\}$ for all positive integers $k\geq 1$. Note that if $k=1$, then we have $\sE_1=\sP$.
%\begin{theorem}[Andrews' series for Rogers-Ramanujan-Gordon $k=3$]\label{thm:A-RRG}
%    We have
%    \begin{align}
%\sum_{\la\in\sR_{3,3}}x^{\ell(\la)}q^{|\la|}&=\sum_{n_1, n_2\geq 0}\frac{x^{n_1+2n_2}q^{\binom{n_1}{2}+4\binom{n_2}{4}+2n_1n_2+n_1+2n_2}}{(q; q)_{n_1}(q^2; q^2)_{n_2}},\\
%\sum_{\la\in\sR_{3,2}}x^{\ell(\la)}q^{|\la|}&=\sum_{n_1, n_2\geq 0}\frac{x^{n_1+2n_2}q^{\binom{n_1}{2}+4\binom{n_2}{4}+2n_1n_2+n_1+4n_2}}{(q; q)_{n_1}(q^2; q^2)_{n_2}},\\
%\sum_{\la\in\sR_{3,1}}x^{\ell(\la)}q^{|\la|}&=\sum_{n_1, n_2\geq 0}\frac{x^{n_1+2n_2}q^{\binom{n_1}{2}+4\binom{n_2}{4}+2n_1n_2+2n_1+4n_2}}{(q; q)_{n_1}(q^2; q^2)_{n_2}}.
%    \end{align}
%\end{theorem}
%\begin{remark}
%Recently, K\i l\i \c{c} give the combinatorial of Theorem \ref{thm:A-RRG} based on the ``base+increments'' framework.
%\end{remark}

%%%%%%%%%%%%%%%%%%%%%%%%%%%%%%%%%%%
\subsection{Partition set $\sS^2$} Recall that $\sS^2$ is indeed the set $\sA_{3, 3}$, i.e., a set of partitions $\la$ where $\la_{j+2}-\la_j\geq 2$. To show that $\sS^2$ is a span one linked partition ideal, we present the following lemma.

%Next, we apply the framework of linked partition ideals to the set $\sR_{3, 3}$, that is, the set of partitions $\la$ where $\la_{i+1}-\la_{i}\geq 2$. For given a partition $\la\in \sP$, we denote the number of parts of $\la$ that are repeated exactly $k$ times as $\ell_k(\la)$. Then we have the following refined theorem.
%\begin{theorem}\label{thm:refine A-RRG}
   % We have
   % \begin{align}  \sum_{\la\in\sR_{3,3}}x^{\ell(\la)}y^{\ell_2(\la)}q^{|\la|}&=\sum_{n_1, n_2\geq 0}\frac{x^{n_1+2n_2}y^{n_2}q^{\binom{n_1}{2}+4\binom{n_2}{4}+2n_1n_2+n_1+2n_2}}{(q; q)_{n_1}(q^2; q^2)_{n_2}},\\
%\sum_{\la\in\sR_{3,2}}x^{\ell(\la)}y^{\ell_2(\la)}q^{|\la|}&=\sum_{n_1, n_2\geq 0}\frac{x^{n_1+2n_2}y^{n_2}q^{\binom{n_1}{2}+4\binom{n_2}{4}+2n_1n_2+n_1+4n_2}}{(q; q)_{n_1}(q^2; q^2)_{n_2}},\\
%\sum_{\la\in\sR_{3,1}}x^{\ell(\la)}y^{\ell_2(\la)}q^{|\la|}&=\sum_{n_1, n_2\geq 0}\frac{x^{n_1+2n_2}y^{n_2}q^{\binom{n_1}{2}+4\binom{n_2}{4}+2n_1n_2+2n_1+4n_2}}{(q; q)_{n_1}(q^2; q^2)_{n_2}}.
  %  \end{align}
%\end{theorem}

\begin{lemma}
The partition set $\sS^{2}$ is a span one linked partition ideal $\sI(\langle\Pi_{\sS^{2}}, \cL_{\sS^{2}}\rangle, 2)$ with $\Pi_{\sS^{2}}=\{\pi_1=\ep, \pi_2=1, \pi_3=1+1, \pi_4=1+2, \pi_5=2, \pi_6=2+2\}$ and
\begin{align*}
    \cL_{\sS^{2}}(\pi_1)&=\{\pi_1, \pi_2, \pi_3, \pi_4, \pi_5, \pi_6\},\\
    \cL_{\sS^{2}}(\pi_2)&=\{\pi_1, \pi_2, \pi_3, \pi_4, \pi_5, \pi_6\},\\
    \cL_{\sS^{2}}(\pi_3)&=\{\pi_1, \pi_2, \pi_3, \pi_4, \pi_5, \pi_6\},\\
    \cL_{\sS^{2}}(\pi_4)&=\{\pi_1, \pi_2, \pi_4, \pi_5, \pi_6\},\\
    \cL_{\sS^{2}}(\pi_5)&=\{\pi_1, \pi_2, \pi_4, \pi_5, \pi_6\},\\
    \cL_{\sS^{2}}(\pi_6)&=\{\pi_1, \pi_5, \pi_6\}.
\end{align*}
\end{lemma}

\begin{proof}
   It is clear that all partitions in $\sI(\langle\Pi_{\sS^2}, \cL_{\sS^2}\rangle, 2)$ satisfy the conditions for $\sS^2$. In fact, we observe that for any partition $\lambda \in \sI(\langle\Pi_{\sS^{2}}, \cL_{\sS^{2}}\rangle, 2)$, the set $\Pi_{\mathscr{S}^{2}}$ together with $T = 2$ ensures that the parts of $\lambda$ appear at most twice. Moreover, $\mathcal{L}_{\mathscr{S}^{2}}$ guarantees that the condition $\lambda_{j+2} - \lambda_j \geq 2$ holds. On the other hand, we decompose each partition in $\sS^{2}$ into blocks $B_0$, $B_1$, ..., where parts between $(2i+1)$ and $(2i+2)$ fall into block $B_i$. Clearly, each block $B_i$ contains at most two parts. In addition, applying $\phi^{-2i}$ to the block $B_i$ yields a partition in $\Pi_{\sS^{2}}$. We then have the following cases:
    \begin{enumerate}[(i)]
        \item if $\phi^{-2i}(B_i)$ is $\pi_1$, $\pi_2$ or $\pi_3$, then $\phi^{-2(i+1)}(B_{i+1})$ can be any partition in $\Pi_{\sS^{2}}$;

        \item if $\phi^{-2i}(B_i)$ is $\pi_4$ or $\pi_5$, then $\phi^{-2(i+1)}(B_{i+1})$ cannot be $\pi_3$. Otherwise, the cases $(2i+1)+(2i+2)+(2i+3)+(2i+3)$ or $(2i+2)+(2i+3)+(2i+3)$ would violate the condition $\la_{j+2}-\la_j\geq 2$ in the definition of $\sS^{2}$;

        \item if $\phi^{-2i}(B_i)$ is $\pi_6$, then $\phi^{-2(i+1)}(B_{i+1})$ cannot be $\pi_2$, $\pi_3$ or $\pi_4$. Otherwise, the cases $(2i+2)+(2i+2)+(2i+3)$, $(2i+2)+(2i+2)+(2i+3)+(2i+3)$ or $(2i+2)+(2i+2)+(2i+3)+(2i+4)$ would violate the condition $\la_{j+2}-\la_j\geq 2$ in the definition of $\sS^{2}$.
    \end{enumerate}
    Therefore, we conclude that $\sS^{2}=\sI(\langle\Pi_{\sS^{2}}, \cL_{\sS^{2}}\rangle, 2)$.
\end{proof}

\begin{example}
Consider the partition $\la=1+1+3+4+6+6+9+12\in \sS^2$, we may decompose it as follows:
\begin{align*}
\la&=\underbrace{\phi^0(1+1)}_{B_0}\oplus\underbrace{\phi^2(1+2)}_{B_1}\oplus\underbrace{\phi^4(2+2)}_{B_2}\oplus\underbrace{\phi^6(\ep)}_{B_3}\oplus\underbrace{\phi^8(1)}_{B_4}\oplus\underbrace{\phi^{10}(2)}_{B_5}\\
&=\phi^0(\pi_3)\oplus\phi^2(\pi_4)\oplus\phi^4(\pi_6)\oplus\phi^6(\pi_1)\oplus\phi^{8}(\pi_2)\oplus\phi^{10}(\pi_5).
\end{align*}
One readily verifies that $\la\in \sI(\langle\Pi_{\sS^{2}}, \cL_{\sS^{2}}\rangle, 2)$.
\end{example}

\begin{proof}[Proof of Theorem \ref{thm:Ak=2}]
Throughout the proof, we retain the block-decomposition \eqref{id:decom} for every partition in $\sS^2$. The generating function of the set $\sS^2$ is subsequently introduced by conditioning on the first block $B_0=\la_0$ of the said decomposition. For $1\leq i\leq 6$ we define: 
\begin{align}
    G_i(x)=G_i(x, y, q):=\sum_{\substack{\la\in\sS^2\\ \la_0=\pi_i}}x^{\ell(\la)}y^{\ell_2(\la)}q^{|\la|}.
\end{align}
By isolating the first block of the partition, it implies that
\begin{align}
    G_i(x)&=x^{\ell(\pi_i)}y^{\ell_2(\pi_i)}q^{|\pi_i|}\sum_{j:\pi_j\in\cL_{\sS^2}(\pi_i)}\sum_{\substack{\la'\in \sS^2\\ \la_0'=\pi_j}}(xq^2)^{\ell(\la')}y^{\ell_2(\la')}q^{|\la'|}\nonumber\\
    &=x^{\ell(\pi_i)}y^{\ell_2(\pi_i)}q^{|\pi_i|}\sum_{j: \pi_j\in \cL_{\sS^{2}}(\pi_i)}G_j(xq^2).
\end{align}
Therefore, we have
\begin{align}
    \left(\begin{array}{cccccccc}
         G_1(x)\\
         G_2(x)\\
         G_3(x)\\
         G_4(x)\\
         G_5(x)\\
         G_6(x)
    \end{array}\right)=\mathrm{W}_{\sS^2}\cdot \mathrm{A}_{\sS^2}\cdot\left(\begin{array}{cccccccc}
         G_1(xq^2)\\
         G_2(xq^2)\\
         G_3(xq^2)\\
         G_4(xq^2)\\
         G_5(xq^2)\\
         G_6(xq^2)
    \end{array}\right),
\end{align}
where
\begin{align}
    \rW_{\sS^2}=\diag(1, xq, x^2yq^2, x^2q^3, xq^2, x^2yq^4)\label{eq:Wk=2}
\end{align}
and
\begin{align}\label{eq:Ak=2}
    \rA_{\sS^2}=\left(\begin{array}{cccccccc}
         1 & 1 & 1 & 1 & 1 & 1\\
         1 & 1 & 1 & 1 & 1 & 1\\
         1 & 1 & 1 & 1 & 1 & 1\\
         1 & 1 & 0 & 1 & 1 & 1\\
         1 & 1 & 0 & 1 & 1 & 1\\
         1 & 0 & 0 & 0 & 1 & 1
    \end{array}\right).
\end{align}
We further write
\begin{align}
    \left(\begin{array}{cccccccc}
         F_1(x)\\
         F_2(x)\\
         F_3(x)\\
         F_4(x)\\
         F_5(x)\\
         F_6(x)
    \end{array}\right)=\rA_{\sS^2}\cdot\left(\begin{array}{cccccccc}
         G_1(x)\\
         G_2(x)\\
         G_3(x)\\
         G_4(x)\\
         G_5(x)\\
         G_6(x)
    \end{array}\right), 
\end{align}
we thus obtain the following system of $q$-difference equations:
\begin{align}
    \left(\begin{array}{cccccccc}
         F_1(x)\\
         F_2(x)\\
         F_3(x)\\
         F_4(x)\\
         F_5(x)\\
         F_6(x)
    \end{array}\right)=\rA_{\sS^2}\cdot \rW_{\sS^2}\cdot \left(\begin{array}{cccccccc}
         F_1(xq^2)\\
         F_2(xq^2)\\
         F_3(xq^2)\\
         F_4(xq^2)\\
         F_5(xq^2)\\
         F_6(xq^2)
    \end{array}\right).\label{id:matrixk=2}
\end{align}
First, observe that $G_1(0)=1$ and $G_i(0)=0$ for $i\in\{2, 3, 4, 5, 6\}$, since the empty partition $\ep$ is counted only in $G_1(x)$. Thus, one can easily check that 
\begin{align*}
    F_1(0)=F_2(0)=F_3(0)=F_4(0)=F_5(0)=F_6(0)=1.
\end{align*}
Recall that each $F_i(x)\in \bC\llbracket q\rrbracket\llbracket x\rrbracket$. If we treat \eqref{id:matrixk=2} as a matrix equation over $\bC\llbracket q\rrbracket\llbracket x\rrbracket$, then \cite[Proposition 15]{che20} ensures the existence and uniqueness of $F_i(x)$. Meanwhile, we note that
\begin{align*}
\sum_{\la\in\sS^{2}}x^{\ell(\la)}y^{\ell_2(\la)}q^{|\la|}&=\sum_{i\in \{1, 2, 3, 4, 5, 6\}}G_i(x)=F_1(x)=F_2(x)=F_3(x),\\
\sum_{\la\in\sS^2_{1}}x^{\ell(\la)}y^{\ell_2(\la)}q^{|\la|}&=\sum_{i\in \{1, 2, 4, 5, 6\}}G_i(x)=F_4(x)=F_5(x),\\
\sum_{\la\in\sS^2_0}x^{\ell(\la)}y^{\ell_2(\la)}q^{|\la|}&=\sum_{i\in \{1, 5, 6\}}G_i(x)=F_6(x).
\end{align*}

Consequently, in order to establish Theorem \ref{thm:Ak=2}, we verify that $H(\ubeta)$ defined below satisfies the same recurrence as \eqref{id:matrixk=2}. Henceforth, we fix $R=J=2$ in \eqref{id:q-multi}, and choose
\begin{align*}
    \ual=\left(\begin{array}{cccccccc}
         1 & 2\\
         2 & 4
    \end{array}\right),\quad \uA=(1,2), \quad \begin{array}{llll} x_1=x,\\x_2=y,\end{array}\quad \text{ and } \begin{array}{lllll}\ugaI=(1,2),\\\ugaII=(0, 1).\end{array}
\end{align*}
We thus obtain
\begin{align}
    H(\ubeta)=H(\beta_1, \beta_2)=\sum_{n_1, n_2\geq 0}\frac{x^{n_1+2n_2}y^{n_2}q^{\binom{n_1}{2}+4\binom{n_2}{2}+2n_1n_2+\beta_1n_1+\beta_2n_2}}{(q; q)_{n_1}(q^2; q^2)_{n_2}}.
\end{align}

We now proceed to prove that 
\begin{align}
    \left(\begin{array}{cccccccc}
         F_1(x)\\
         F_2(x)\\
         F_3(x)\\
         F_4(x)\\
         F_5(x)\\
         F_6(x)
    \end{array}\right)=\left(\begin{array}{cccccccc}
         H(1, 2)\\
         H(1, 2)\\
         H(1, 2)\\
         H(1, 4)\\
         H(1, 4)\\
         H(2, 4)\\
    \end{array}\right).\label{id:FHk=2}
\end{align}
It suffices to show that
\begin{align}
    \left(\begin{array}{cccccccc}
         H(1, 2)\\
         H(1, 2)\\
         H(1, 2)\\
         H(1, 4)\\
         H(1, 4)\\
         H(2, 4)\\
    \end{array}\right)=\rA_{\sS^2}\cdot \rW_{\sS^2}\cdot \left(\begin{array}{cccccccc}
         H(3, 6)\\
         H(3, 6)\\
         H(3, 6)\\
         H(3, 8)\\
         H(3, 8)\\
         H(4, 8)\\
    \end{array}\right).
\end{align}
where $\rW_{\sS^2}$ and $\rA_{\sS^2}$ are the same as \eqref{eq:Wk=2} and \eqref{eq:Ak=2}, respectively. If \eqref{id:FHk=2} holds, then Theorem \ref{thm:Ak=2} follows from the preceding discussion. We make use of Lemma \ref{lem:key rec} repeatedly. Recall that the coordinate on which the recurrence \eqref{id:key rec} acts is denoted by $\check{\beta}$. Then we have
\begin{align*}
H(1, \check{2})&=H(\check{1}, 4)+x^2yq^2H(3, 6)\\
&=(H(2, \check{4})+xqH(\check{2}, 6))+x^2yq^2H(3, 6)\\
&=(H(\check{2}, 6)+x^2yq^4H(4, 8))+xq(H(3, 6)+xq^2H(3, 8))+x^2yq^2H(3, 6)\\
&=(H(3, 6)+xq^2H(3, 8))+x^2yq^4H(4, 8)+xqH(3, 6)+x^2q^3H(3, 8))+x^2yq^2H(3, 6)\\
&=(1+xq+x^2yq^2)H(3, 6)+(x^2q^3+xq^2)H(3, 8)+x^2yq^4H(4, 8).
\end{align*}
In fact, the foregoing derivation simultaneously yields that
\begin{align*}
H(1, 4)&=(1+xq)H(3, 6)+(x^2q^3+xq^2)H(3, 8)+x^2yq^4H(4, 8),\\
H(2, 4)&=H(3, 6)+xq^2H(3, 8)+x^2yq^4H(4, 8).
\end{align*}
This procedure is encoded in the binary tree shown in Figure \ref{fig:k2}. Therefore, the proof is complete.
%Then starting from $H(1, 2)$, we repeatedly apply Lemma \ref{lem:key rec} in (2, 1) order, namely, for $H(\beta_1, \beta_2)$, we apply Lemma \ref{lem:key rec} with respect to $\beta_2$ and $\beta_1$ in sequence. The whole binary tree is shown in Figure \ref{fig:k2}. Hence, based on the process, we drive the following relation.
\begin{figure}[htbp]
\begin{center}
\begin{tikzpicture}[scale=1.2,
  % 顶点：只有圆圈，没有外框
  every node/.style={
  circle,
  minimum size=5mm,
  inner sep=0pt,
  text width=8mm,
    align=center},
  every edge/.style={draw,thick}]
  
  % 顶点
  \node (1) at (0,0) {$H$$(1, \check{2})$};
  \node (2) at (-1.5,-1) {$H$$(\check{1}, 4)$};
  \node (3) at (1.5,-1) {$H$$(3, 6)$};
  \node (4) at (-3,-2) {$H$$(2, \check{4})$};
  \node (5) at (0.5,-2) {$\boxed{H(\check{2}, 6)}$};
  \node (6) at (-4,-3) {$H$$(\check{2}, 6)$};
  \node (7) at (-2,-3) {$H$$(4, 8)$};
  \node (8) at (-5,-4) {$H$$(3, 6)$};
  \node (9) at (-3,-4) {$H$$(3, 8)$};
  \node (10) at (-0.5,-3) {$H$$(3, 6)$};
  \node (11) at (1.5,-3) {$H$$(3, 8)$};
  % 边（生成树），权重不画圈
  \path
    (1) edge node[left]{1} (2)
    (1) edge node[right]{$x^2yq^2$} (3)
    (2) edge node[left]{1} (4)
    (2) edge node[right]{$xq$} (5)
    (4) edge node[left]{1} (6)
    (4) edge node[right]{$x^2yq^4$} (7)
    (5) edge node[left]{1} (10)
    (5) edge node[right]{$xq^2$} (11)
    (6) edge node[left]{1} (8)
    (6) edge node[right]{$xq^2$} (9);
\end{tikzpicture}
\end{center}
\caption{The binary tree for $\sS^{2}$.}
\label{fig:k2}
\end{figure}
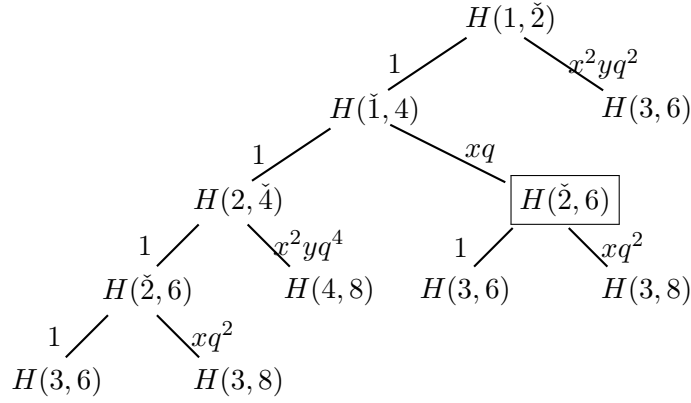

 %Hence, we have

%Finally, note that
%\begin{align}
%    F_1(x)=\sum_{\la\in\sR_{3,3}}x^{\ell(\la)}y^{\ell_2(\la)}q^{|\la|},\ F_4(x)=\sum_{\la\in\sR_{3,2}}x^{\ell(\la)}y^{\ell_2(\la)}q^{|\la|},\ F_6(x)=\sum_{\la\in\sR_{3,1}}x^{\ell(\la)}y^{\ell_2(\la)}q^{|\la|}.
%\end{align}
\end{proof}

\begin{remark}\label{rem:k2}
As shown in Figure \ref{fig:k2}, every desired solution occurs at one of the leftmost nodes in the binary tree; for these nodes, any right child that could otherwise be further decomposed is in fact already contained in these nodes. For  instance, $H(2, 6)$, framed as a right child, appears in the second-to-last position of the leftmost nodes. Moreover, along the leftmost branch, the lemma is applied repeatedly by first targeting the second coordinate and then the first, i.e., in the order $(2, 1)$, and this pair of steps is carried out twice in succession.
\end{remark}

%%%%%%%%%%%%%%%%%%%%%%%%%%%%%%%%%%%%
\subsection{Partition set $\sS^3$} Recall that $\sS^3$ is a set of partitions where each part appears at most three times and satisfying the following gap condition
\begin{align*}
\begin{pNiceMatrix}[first-row, first-col, nullify-dots, columns-width=auto]
 & [b]_1 & [b]_2 & [b]_3\\
 [a]_1 & 1 & 2 & 2\\
 [a]_2 & 2 & 2 & 3\\
 [a]_3 & 3 & 3 & 3
\end{pNiceMatrix}
\end{align*}
where the $(i, j)$-entry gives the minimal allowed value of $(b-a)$ for consecutive tuples $[a]_i$ and $[b]_j$ with $a<b$. Then to show that $\sS^3$ is a span one linked partition ideal, we provide the following lemma.

%Next, in order to see the second example, we need some concepts. Given a partition $\la\in \sP$ and a positive integer $a$, if there are exactly $k$ parts of $\la$ equal to $a$ then we can define the $k$-tuple $$[a]_k=\underbrace{a+a+\cdots +a}_{k}.$$ Now we see that $\sR_{3, 3}$ can be redefined as the set of partitions $\la$ where each part appears at most twice and satisfying the following gap condition
%\begin{align}
%\begin{pNiceMatrix}[first-row, first-col, nullify-dots, columns-width=auto]
 %& [b]_1 & [b]_2\\
 %[a]_1 & 1 & 2\\
% [a]_2 & 2 & 2
%\end{pNiceMatrix}
%\end{align}
%where $[a]_i$ and $[b]_j$ are two consecutive distinct tuples and each element in this matrix is the minimum difference between $a$ and $b$ for $a<b$ and $i, j=1, 2$. Therefore, naturally we may consider the set $\sS^3$ of partitions where 

\begin{lemma}
The partition set $\sS^3$ is a span one linked partition ideal $\sI(\langle \Pi_{\sS^3}, \cL_{\sS^3}\rangle, 3)$ with 
\begin{align*}
\Pi_{\sS^3}=\left\{\begin{array}{llll}
\pi_1=\ep,\ \pi_2=1,\ \pi_3=2,\ \pi_4=1+1,\ \pi_5=1+2,\ \pi_6=1+1+1, \\
\pi_7=2+2,\\
 \pi_8=3,\ \pi_9=1+3,\ \pi_{10}=2+3,\ \pi_{11}=1+1+3,\ \pi_{12}=1+2+3,\\
 \pi_{13}=2+2+2,\\
 \pi_{14}=3+3,\ \pi_{15}=1+3+3,\ \pi_{16}=1+1+3+3,\\
 \pi_{17}=3+3+3,\ \pi_{18}=1+3+3+3
\end{array}\right\}
\end{align*}
and 
\begin{align*}
\cL_{\sS^3}(\pi_i)=\begin{cases}
%\bigcup_{1\leq k\leq 18}\{\pi_k\} 
\Pi_{\sS^3} & \text{ if }i=1, 2, 3, 4, 5, 6,\\
\Pi_{\sS^3}\backslash\{\pi_6\} 
%\{\pi_1, \pi_2, \pi_3, \pi_4, \pi_5, \pi_7, \pi_8, \pi_9, \pi_{10}, \pi_{11}, \pi_{12}, \pi_{13}, \pi_{14}, \pi_{15}, \pi_{16}, \pi_{17}, \pi_{18}\}
%\bigcup_{1\leq k\leq 5, 7\leq k\leq 18}\{\pi_k\} 
& \text{ if }i=7,\\
\Pi_{\sS^3}\backslash \{\pi_4, \pi_6, \pi_{11}, \pi_{16}\} & \text{ if }i=8, 9, 10, 11, 12,\\
\{\pi_1, \pi_3, \pi_7, \pi_8, \pi_{10}, \pi_{13}, \pi_{14}, \pi_{17}\} & \text{ if }i=13,\\
\{\pi_1, \pi_3, \pi_7, \pi_8, \pi_{10}, \pi_{14}, \pi_{17}\} & \text{ if }i=14, 15, 16,\\
\{\pi_1, \pi_8, \pi_{14}, \pi_{17}\} &\text{ if }i=17, 18.
\end{cases}
\end{align*}
\end{lemma}

\begin{proof}
It is clear that all partitions in $\sI(\langle\Pi_{\sS^3}, \cL_{\sS^3}\rangle, 3)$ satisfy the conditions for $\sS^3$. On the other hand, we decompose each partition in $\sS^3$ into blocks $B_0, B_1, ...$ such that all parts between $3i+1$ and $3i+3$ fall into block $B_i$. It is plain that $\phi^{-3i}(B_i)$ is exclusively from $\Pi_{\sS^3}$. Moreover, if $\phi^{-3i}(B_i)$ is $\pi_i$ for $i=1, 2, 3, 4, 5, 6$, then clearly $\phi^{-3(i+1)}(B_{i+1})$ can be any among $\Pi_{\sS^3}$. If $\phi^{-3i}(B_i)$ is $\pi_7$ so that $B_i$ is $(3i+2)+(3i+2)$, then $B_{i+1}$ cannot be $\pi_6$; otherwise, the parts $(3i+2)+(3i+2)+(3i+4)+(3i+4)+(3i+4)$ would violate the difference conditions in $\sS^3$. One may carry out similar arguments for other possibilities of $\phi^{-3i}(B_i)$ and the details are omitted.
\end{proof}

\begin{example}
Consider the partition $\la=1+2+3+4+7+7+7+10+12+12+12+17+18\in \sS^3$, we may decompose it as follows:
\begin{align*}
\la&=\underbrace{\phi^0(1+2+3)}_{B_0}\oplus\underbrace{\phi^3(1)}_{B_1}\oplus\underbrace{\phi^6(1+1+1)}_{B_2}\oplus\underbrace{\phi^9(1+3+3+3)}_{B_3}\oplus\underbrace{\phi^{12}(\ep)}_{B_4}\oplus\underbrace{\phi^{15}(2+3)}_{B_5}\\
&=\phi^0(\pi_{12})\oplus\phi^3(\pi_2)\oplus\phi^6(\pi_6)\oplus\phi^9(\pi_{18})\oplus\phi^{12}(\pi_1)\oplus\phi^{15}(\pi_{10}).
\end{align*}
One readily verifies that $\la\in \sI(\langle\Pi_{\sS^{3}}, \cL_{\sS^{3}}\rangle, 3)$.
\end{example}

\begin{proof}[Proof of Theorem \ref{thm:Ak=3}]
There are arguments similar to those in the proof of Theorem \ref{thm:Ak=2}, so we may omit some details that are straightforward to verify. For $1\leq i\leq 18$, we define:
\begin{align*}
G_i(x)=G_i(x, y, z, q):&=\sum_{\substack{\la\in \sS^3\\ \la_0=\pi_i}}x^{\ell(\la)}y^{\ell_2(\la)}z^{\ell_3(\la)}q^{|\la|}\\
&=x^{\ell(\pi_i)}y^{\ell_2(\pi_i)}z^{\ell_3(\pi_i)}q^{|\pi_i|}\sum_{j: \pi_j\in \cL_{\sS^3}(\pi_i)}G_j(xq^3).
\end{align*}
Then 
\begin{align}
\left(\begin{array}{cccccccccccc}
G_1(x)\\
G_2(x)\\
\vdots\\
G_{18}(x)
\end{array}\right)=\rW_{\sS^3}\cdot \rA_{\sS^3}\cdot \left(\begin{array}{cccccccccccc}
G_1(xq^3)\\
G_2(xq^3)\\
\vdots\\
G_{18}(xq^3)
\end{array}\right)
\end{align}
where
\begin{align}
\rW_{\sS^3}=\diag(1, xq, &xq^2, x^2yq^2, x^2q^3, x^3zq^3, x^2yq^4, xq^3,\nonumber \\
&x^2q^4, x^2q^5, x^3yq^5, x^3q^6, x^3zq^6, x^2yq^6, x^3yq^7, x^4y^2q^8, x^3zq^9, x^4zq^{10})
\end{align}
and 
\begin{align}
\rA_{\sS^3}=\left(\begin{array}{cccccccccccccccccccccccccccccccccccccccccc}
1 & 1 & 1 & 1 & 1 & 1 & 1 & 1 & 1 & 1 & 1 & 1 & 1 & 1 & 1 & 1 & 1 & 1\\
1 & 1 & 1 & 1 & 1 & 1 & 1 & 1 & 1 & 1 & 1 & 1 & 1 & 1 & 1 & 1 & 1 & 1\\
1 & 1 & 1 & 1 & 1 & 1 & 1 & 1 & 1 & 1 & 1 & 1 & 1 & 1 & 1 & 1 & 1 & 1\\
1 & 1 & 1 & 1 & 1 & 1 & 1 & 1 & 1 & 1 & 1 & 1 & 1 & 1 & 1 & 1 & 1 & 1\\
1 & 1 & 1 & 1 & 1 & 1 & 1 & 1 & 1 & 1 & 1 & 1 & 1 & 1 & 1 & 1 & 1 & 1\\
1 & 1 & 1 & 1 & 1 & 1 & 1 & 1 & 1 & 1 & 1 & 1 & 1 & 1 & 1 & 1 & 1 & 1\\
1 & 1 & 1 & 1 & 1 & 0 & 1 & 1 & 1 & 1 & 1 & 1 & 1 & 1 & 1 & 1 & 1 & 1\\
1 & 1 & 1 & 0 & 1 & 0 & 1 & 1 & 1 & 1 & 0 & 1 & 1 & 1 & 1 & 0 & 1 & 1\\
1 & 1 & 1 & 0 & 1 & 0 & 1 & 1 & 1 & 1 & 0 & 1 & 1 & 1 & 1 & 0 & 1 & 1\\
1 & 1 & 1 & 0 & 1 & 0 & 1 & 1 & 1 & 1 & 0 & 1 & 1 & 1 & 1 & 0 & 1 & 1\\
1 & 1 & 1 & 0 & 1 & 0 & 1 & 1 & 1 & 1 & 0 & 1 & 1 & 1 & 1 & 0 & 1 & 1\\
1 & 1 & 1 & 0 & 1 & 0 & 1 & 1 & 1 & 1 & 0 & 1 & 1 & 1 & 1 & 0 & 1 & 1\\
1 & 0 & 1 & 0 & 0 & 0 & 1 & 1 & 0 & 1 & 0 & 0 & 1 & 1 & 0 & 0 & 1 & 0\\
1 & 0 & 1 & 0 & 0 & 0 & 1 & 1 & 0 & 1 & 0 & 0 & 0 & 1 & 0 & 0 & 1 & 0\\
1 & 0 & 1 & 0 & 0 & 0 & 1 & 1 & 0 & 1 & 0 & 0 & 0 & 1 & 0 & 0 & 1 & 0\\
1 & 0 & 1 & 0 & 0 & 0 & 1 & 1 & 0 & 1 & 0 & 0 & 0 & 1 & 0 & 0 & 1 & 0\\
1 & 0 & 0 & 0 & 0 & 0 & 0 & 1 & 0 & 0 & 0 & 0 & 0 & 1 & 0 & 0 & 1 & 0\\
1 & 0 & 0 & 0 & 0 & 0 & 0 & 1 & 0 & 0 & 0 & 0 & 0 & 1 & 0 & 0 & 1 & 0
\end{array}\right).
\end{align}
Let 
\begin{align}
\left(\begin{array}{cccccccccccc}
F_1(x)\\
F_2(x)\\
\vdots\\
F_{18}(x)
\end{array}\right)=\rA_{\sS^3}\cdot \left(\begin{array}{cccccccccccc}
G_1(x)\\
G_2(x)\\
\vdots\\
G_{18}(x)
\end{array}\right).
\end{align}
Then we have the system of $q$-difference equations as follows:
\begin{align}
\left(\begin{array}{cccccccccccc}
F_1(x)\\
F_2(x)\\
\vdots\\
F_{18}(x)
\end{array}\right)=\rA_{\sS^3}\cdot\rW_{\sS^3}\cdot \left(\begin{array}{cccccccccccc}
F_1(xq^3)\\
F_2(xq^3)\\
\vdots\\
F_{18}(xq^3)
\end{array}\right).\label{id:k3}
\end{align}

Note that the solution of \eqref{id:k3} is uniquely determined by $F_1(0)=\cdots =F_{18}(0)=1$. And recall that $\sS^3_{i, j}$ be the set of partitions in $\sS^3$ where at most $i$ of the parts are equal to $1$ and at most $j$ of the parts are equal to $2$ for $0\leq i, j\leq 3$ (note that $\sS^3_{3, 3}=\sS^3$). Then we have

\begin{align*}
\sum_{\la\in\sS^3}x^{\ell(\la)}y^{\ell_2(\la)}z^{\ell_3(\la)}q^{|\la|}&=\sum_{1\leq i\leq 18}G_{i}(x)=F_1(x)=\cdots =F_{6}(x),\\
\sum_{\la\in\sS^3_{2, 3}}x^{\ell(\la)}y^{\ell_2(\la)}z^{\ell_3(\la)}q^{|\la|}&=\sum_{\substack{1\leq i\leq 18\\ i\neq 6}}G_i(x)=F_7(x),\\
\sum_{\la\in\sS^3_{1, 3}}x^{\ell(\la)}y^{\ell_2(\la)}z^{\ell_3(\la)}q^{|\la|}&=\sum_{\substack{1\leq i\leq 18\\ i\notin \{4, 6, 11, 16\}}}G_i(x)=F_8(x)=\cdots =F_{12}(x),\\
\sum_{\la\in\sS^3_{0, 3}}x^{\ell(\la)}y^{\ell_2(\la)}z^{\ell_3(\la)}q^{|\la|}&=\sum_{i\in \{1, 3, 7, 8, 10, 13, 14, 17\}}G_i(x)=F_{13}(x),\\
\sum_{\la\in\sS^3_{0, 2}}x^{\ell(\la)}y^{\ell_2(\la)}z^{\ell_3(\la)}q^{|\la|}&=\sum_{i\in\{1, 3, 7, 8, 10, 14, 17\}}G_i(x)=F_{14}(x)=F_{15}(x)=F_{16}(x),\\
\sum_{\la\in\sS^3_{0, 0}}x^{\ell(\la)}y^{\ell_2(\la)}z^{\ell_3(\la)}q^{|\la|}&=\sum_{i\in\{1, 8, 14, 17\}}G_i(x)=F_{17}(x)=F_{18}(x).
\end{align*}

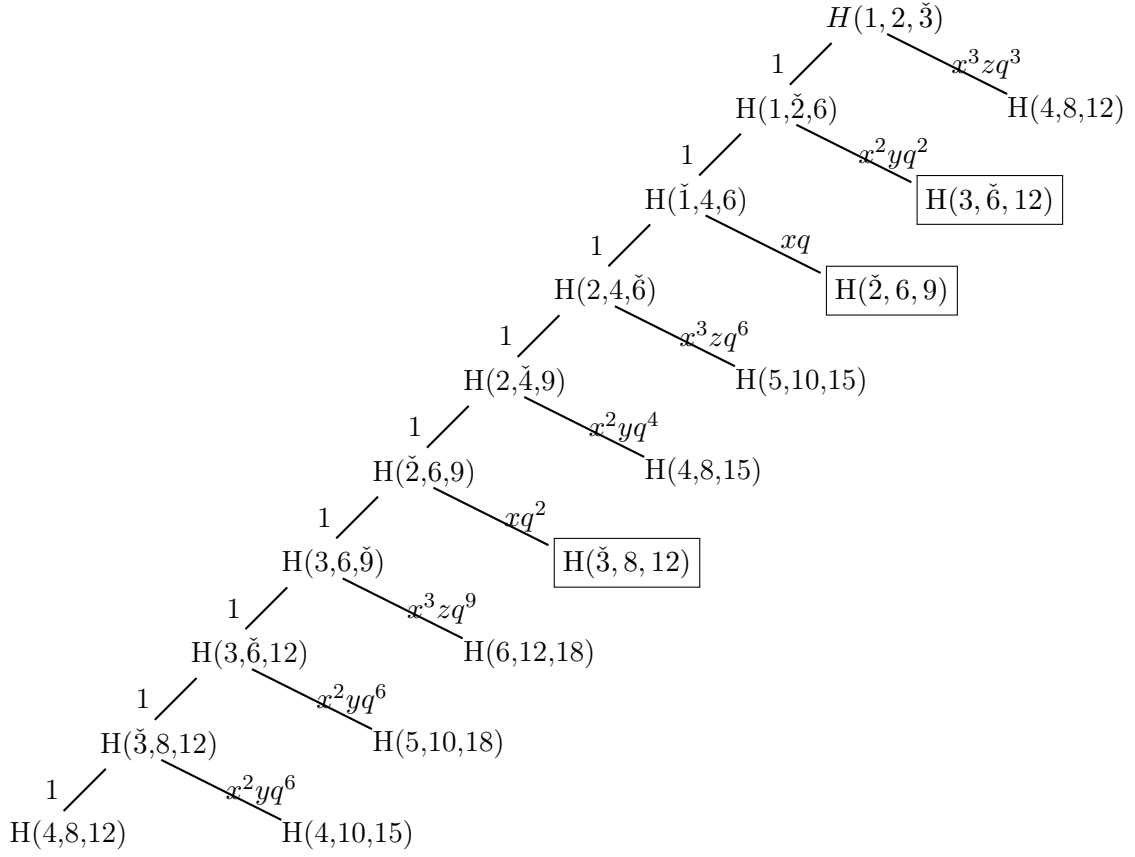
\begin{figure}[htbp]
\begin{center}
\begin{tikzpicture}[scale=1.2,
  % 顶点：只有圆圈，没有外框
  every node/.style={
  circle,
  minimum size=5mm,
  inner sep=0pt,
  text width=8mm,
    align=center        % 居中对齐
    },
  every edge/.style={draw,thick}]
  
  % 顶点
  %第一层
  \node (1) at (0,0) {$H$$(1, 2, \check{3})$};
  %第二层
  \node (2) at (-1,-1) {H(1,$\check{2}$,6)};
  \node (3) at (2,-1) {H(4,8,12)};
  %第三层
  \node (4) at (-2,-2) {H($\check{1}$,4,6)};
  \node (5) at (1,-2) {$\boxed{\mathrm{H}(3,\check{6}, 12)}$};
  %第三层
  \node (6) at (-3,-3) {H(2,4,$\check{6}$)};
  \node (7) at (0,-3) {$\boxed{\mathrm{H}(\check{2},6,9)}$};
  %第四层
  \node (8) at (-4,-4) {H(2,$\check{4}$,9)};
  \node (9) at (-1,-4) {H(5,10,15)};
  %第5层
  \node (10) at (-5,-5) {H($\check{2}$,6,9)};
  \node (11) at (-2,-5) {H(4,8,15)};
  %第6层
  \node (12) at (-6,-6) {H(3,6,$\check{9}$)};
  \node (13) at (-3,-6) {$\boxed{\mathrm{H}(\check{3},8,12)}$};
  %第7层
  \node (14) at (-7,-7) {H(3,$\check{6}$,12)};
  \node (15) at (-4,-7) {H(6,12,18)};
  %第8层
  \node (16) at (-8,-8) {H($\check{3}$,8,12)};
  \node (17) at (-5,-8) {H(5,10,18)};
  %第9层
  \node (18) at (-9,-9) {H(4,8,12)};
  \node (19) at (-6,-9) {H(4,10,15)};
  % 边（生成树），权重不画圈
  \path
  %第一层
    (1) edge node[left]{1} (2)
    (1) edge node[right]{$x^3zq^3$} (3)
    %第二层
    (2) edge node[left]{1} (4)
    (2) edge node[right]{$x^2yq^2$} (5)
    %第三层
    (4) edge node[left]{1} (6)
    (4) edge node[right]{$xq$} (7)
    %第四层
    (6) edge node[left]{1} (8)
    (6) edge node[right]{$x^3zq^6$} (9)
    %第5层
    (8) edge node[left]{1} (10)
    (8) edge node[right]{$x^2yq^4$} (11)
    %第6层
    (10) edge node[left]{1} (12)
    (10) edge node[right]{$xq^2$} (13)
    %第7层
    (12) edge node[left]{1} (14)
    (12) edge node[right]{$x^3zq^9$} (15)
    %第8层
    (14) edge node[left]{1} (16)
    (14) edge node[right]{$x^2yq^6$} (17)
    %第8层
    (16) edge node[left]{1} (18)
    (16) edge node[right]{$x^2yq^6$} (19);
\end{tikzpicture}
\end{center}
\caption{The binary tree for $\sS^3$.}
\label{fig:kf}
\end{figure}

Next, to obtain the solution of \eqref{id:k3}, we fix $R=J=3$ in \eqref{id:q-multi}, and choose
\begin{align*}
\ual=\left(\begin{array}{ccccc}
1 & 2 & 3\\
2 & 4 & 6\\
3 & 6 & 9
\end{array}\right),\quad \uA=(1, 2, 3),\quad\begin{array}{lllll} \uga_1=(1, 2, 3),\\ \uga_2=(0, 1, 0),\\ \uga_3=(0, 0, 1),\end{array}\quad\begin{array}{lllll} x_1=x,\\ x_2=y,\\ x_3=z.\end{array}
\end{align*}
Hence, we obtain
\begin{align}
H(\ubeta)=H(\beta_1, \beta_2, \beta_3)=\sum_{n_1, n_2, n_3\geq 0}&\frac{x^{n_1+2n_2+3n_3}y^{n_2}z^{n_3}}{(q; q)_{n_1}(q^2; q^2)_{n_2}(q^3; q^3)_{n_3}}\nonumber\\
&\times q^{\binom{n_1}{2}+4\binom{n_2}{2}+9\binom{n_3}{2}+2n_1n_2+3n_1n_3+6n_2n_3+\beta_1n_1+\beta_2n_2+\beta_3n_3}.
\end{align}

Finally, by iteratively applying Lemma \ref{lem:key rec} in the order $(3, 2, 1)$ (namely, for $H(\beta_1, \beta_2, \beta_3)$, we apply Lemma \ref{lem:key rec} with respect to $\beta_3$, $\beta_2$ and $\beta_1$ in sequence), we obtain the binary tree shown in Figure \ref{fig:kf}. Observe that every node enclosed by a square box can be further decomposed, and the corresponding decomposition is encoded in the leftmost nodes of the tree. For instance, the node $H(3, 8, 12)$ re-appears in the second-to-last leftmost node. This derivation further implies that the solution of matrix equation \eqref{id:k3} equals
\begin{align*}
F_i(x)=\begin{cases}
H(1, 2, 3) & \text{ if }i=1, 2, 3, 4, 5, 6,\\
H(1, 2, 6) & \text{ if }i=7,\\
H(1, 4, 6) & \text{ if }i=8, 9, 10, 11, 12,\\
H(2, 4, 6) & \text{ if }i=13,\\
H(2, 4, 9) & \text{ if }i=14, 15, 16,\\
H(3, 6, 9) & \text{ if }i=17, 18,\\
\end{cases}
%F_1(x)=\cdots =F_6(x)&=H(1, 2, 3),\\
%F_7(x)&=H(1, 2, 6),\\
%F_8(x)=\cdots =F_{12}(x)&=H(1, 4, 6),\\
%F_{13}(x)&=H(2, 4, 6),\\
%F_{14}(x)=F_{15}(x)=F_{16}(x)&=H(2, 4, 9),\\
%F_{17}(x)=F_{18}(x)&=H(3, 6, 9).
\end{align*}
which is exactly the desired result.

%Then starting from $H(1, 2, 3)$, we repeatedly apply Lemma \ref{lem:key rec} in $(3, 2, 1)$ order. The whole binary tree is shown in Figure \ref{fig:k4}. Hence, based on the process, we drive the following relation.

\end{proof}

\begin{remark}
As discussed in the two preceding proofs, for any $k$ we can incorporate $\sS^k$ into the framework of span one linked partition ideals. In fact, we fix $R=J=k$ in \eqref{id:q-multi}, and choose
\begin{align*}
\ual=\left(\begin{array}{cccccccc}
1 & 2 & \cdots & k\\
2 & 4 & \cdots & 2k\\
\vdots & \vdots & & \vdots\\
k & 2k & \cdots & k^2
\end{array}\right),\quad \uA=(1, 2, ..., k),\quad \begin{array}{llll}\uga_1=(1, 2, ..., k),\\ \uga_i=(0, ..., 0, 1, 0, ..., 0) \text{ for }2\leq i\leq k.\end{array}
\end{align*}
where $\uga_i$ is a $k$-dimensional vector with a $1$ in the $i$-th position and $0$'s elsewhere. Hence, we have
\begin{align*}
H(1, 2, ..., k)=\sum_{n_1, ..., n_k\geq 0}&\frac{x_1^{\sum_{i=1}^k in_i}(\prod_{i=2}^kx_i^{n_i})q^{\sum_{i=1}^k(i^2\binom{n_i}{2}+in_i)+\sum_{1\leq i<j\leq k}ijn_in_j}}{(q; q)_{n_1}(q^2; q^2)_{n_2}\cdots (q^k; q^k)_{n_k}}.
\end{align*}
Interested readers may work through the case for arbitrary $k$ step-by-step by using the binary tree representation. In the next section, we will present a purely combinatorial proof for $\sS^k$.
%As explained in Remark \ref{rem:k2} and illustrated in the preceding proof, this specific decomposition ensures that the required solutions for $\sS^2$ and $\sS^3$ are located precisely at the leftmost nodes of the resulting binary tree (including the root). Any right node that can be decomposed further re-appears as a leftmost node at a deeper level. 
\end{remark}

%%%%%%%%%%%%%%%%%%%%%%%%%%%%%%%%%%%
\section{Combinatorics of the partition set $\sS^k$}\label{sec:com-k}

In this section, we focus on the partition set $\sS^k$ introduced in Definition \ref{def:sk}, and provide a combinatorial proof of its refined generating function, as stated in Theorem \ref{thm:Ak}. For $n_1, ..., n_k\geq 0$, we define
\begin{align*}
\sS^k_{\bn}=\sS^k_{(n_1, n_2, ..., n_k)}:=\{\la\in \sS^k: \ell(\la)=\sum_{i=1}^kin_i, \text{ and }\ell_i(\la)=n_i\text{ for }2\leq i\leq k\}.
\end{align*}
And then we may define the base partition for $\sS^k_{\bn}$:
%the exponent of $q$ on the right hand side of \eqref{id:Ak} can be decomposed as follows:
\begin{align}
\beta^{\bn}=\beta^{(n_1, n_2, ..., n_k)}:=\beta^{(n_k)}\oplus \cdots \oplus \beta^{(n_2)}\oplus\beta^{(n_1)},\label{id:base}
%&[1]_k+[k+1]_k+\cdots +[k(n_k-1)+1]_k\\
%&+\cdots+[\sum_{j=i}^{k-1}(j+1)n_{j+1}+1]_i+\cdots +[\sum_{j=i}^{k-1}(j+1)n_{j+1}+i(n_i-1)+1]_i\\
%&+\cdots +[\sum_{j=1}^{k-1}(j+1)n_{j+1}+1]_1+\cdots +[\sum_{j=1}^{k-1}(j+1)n_{j+1}+n_1]_1.
\end{align}
where for any $1\leq i\leq k$ we have
\begin{align*}
\beta^{(n_i)}=\left[\sum_{j=i+1}^{k}(jn_j)+1\right]_i+\left[\sum_{j=i+1}^{k}(jn_j)+i+1\right]_i+\cdots +\left[\sum_{j=i+1}^{k}(jn_{j})+i(n_i-1)+1\right]_i.
\end{align*}
Notice that if we take the first part of $\beta^{(i-1)}$ and the last part of $\beta^{(i)}$, the difference is
\begin{align*}
\left(\sum_{j=i}^{k}(jn_j)+1\right)-\left(\sum_{j=i+1}^{k}(jn_j)+i(n_i-1)+1\right)=i,
\end{align*}
so the $\beta^{(i)}$ are pairwise disjoint. Furthermore, one can verify that $\beta^{\bn}\in \sS^k$. Next, the weight and the partition statistics are easily seen to be as follows:
\begin{align*}
&|\beta^{\bn}|=\sum_{i=1}^k(i^2\binom{n_i}{2}+in_i)+\sum_{1\leq i<j\leq k}ijn_in_j,\\
&\ell(\beta^{\bn})=\sum_{i=1}^{k}in_i,\quad \ell_i(\beta^{\bn})=n_i\text{ for all }2\leq i\leq k.
\end{align*}
Finally, one can check that $\beta^{\bn}$ is the base partition for $\sS^k_{\bn}$, i.e., it is the unique partition in $\sS^k_{\bn}$ with the smallest weight. So far we have provided a combinatorial interpretation of the numerator on the right hand side of \eqref{id:Ak}. Next, for the denominator, we introduce $k$ incremental partitions. First, let $\sE_k$ be the set of partitions with parts which are multiples of $k$ and set $\sE_{k ,n}:=\{\la\in \sE_k: \ell(\la)\leq n\}$ for all positive integers $k\geq 1$. Note that if $k=1$, then we have $\sE_1=\sP$ which is the set of all partitions. It is readily seen that each factor $\frac{1}{(q; q)_i}$ generates precisely the incremental partitions $\mu^i\in \sE_{i, n_i}$ for $1\leq i\leq k$. Consequently, we have
\begin{align*}
\sum_{\substack{n_1, n_2, ..., n_k\geq 0\\ \beta^{\bn}, \mu^1, \mu^2, ..., \mu^k}}x_1^{\ell(\beta^{\bn})}(\prod_{i=2}^kx_i^{\ell_i(\beta^{\bn})})q^{|\beta^{\bn}|+\sum_{i=1}^{k}|\mu^i|}=\sum_{n_1, ..., n_k\geq 0}&\frac{x_1^{\sum_{i=1}^k in_i}(\prod_{i=2}^kx_i^{n_i})}{(q; q)_{n_1}(q^2; q^2)_{n_2}\cdots (q^k; q^k)_{n_k}}\\
&\times q^{\sum_{i=1}^k(i^2\binom{n_i}{2}+in_i)+\sum_{1\leq i<j\leq k}ijn_in_j}.
\end{align*}
From the preceding discussion, by appending these $\mu^i$ to the base partition $\beta^{\bn}$ we shall obtain all partitions in the set $\sS^k_{\bn}$. To make this claim explicit, we now construct the following bijection.

\begin{lemma}
For $\bn=(n_1, n_2, ..., n_k)$ and $n_1, n_2, ..., n_k\geq 0$, there exists a bijection
\begin{align*}
\varphi=\varphi_{\bn}:\left\{\beta^{\bn}\right\}\times \sE_{1, n_1}\times\sE_{2, n_2}\times\cdots\times\sE_{k, n_k}&\ri\sS^k_{\bn}\\
(\beta^{\bn}, \mu^1, \mu^2, ..., \mu^k)&\mapsto\la
\end{align*}
such that $|\la|=|\beta^{\bn}|+\sum_{i=1}^k|\mu^i|$, $\ell(\la)=\ell(\beta^{\bn})$ and $\ell_i(\la)=\ell_i(\beta^{\bn})$ for $2\leq i\leq k$. Consequently, Theorem \ref{thm:Ak} follows.
\end{lemma}

\begin{proof}
For a given $(k+1)$-fold partition vector $(\beta^{\bn}, \mu^1, ..., \mu^k)$ in which $\beta^{\bn}$ is the base partition as described in \eqref{id:base} and each $\mu^i=\mu_1^i+\mu_2^i+\cdots +\mu_{n_i}^i\in \sE_{i, n_i}$ for $1\leq i\leq k$ (if $\ell(\mu^i)<n_i$ then we pad the deficient positions with zeros, noting that these zeros do not participate in the calculation of any statistics), we first define the map $\varphi$. Recall that $\beta^{\bn}=\beta^{(n_k)}\oplus \cdots \oplus \beta^{(n_2)}\oplus\beta^{(n_1)}$. Then the map processes $i$ from $1$ to $k$; for each fixed $i$ it successively applies $\mu_j^i/i$ forward moves to the $j$-th $i$-tuple of $\beta^{(i)}$, and $j$ running backwards from $n_i$ to $1$. Throughout this process, a sequence of adjustments is applied to finally produce the partition $\la\in \sS^k_{\bn}$. We introduce the two operations required for this map: {\it forward moves} and {\it adjustments}.
\begin{description}
\item[Forward moves] For any $i$-tuple $[t]_i$, one forward move transforms it into $[t+1]_i$, increasing its weight by $i$. 

\item[Adjustments] As noted above, forward moves do not always proceed without difficulty; obstructions may rise because the partition obtained after the forward move would violate the difference condition defining the set $\sS^k_{\bn}$, necessitating the introduction of an adjustment. Consider an $i$-tuple $[t]_i$ that is scheduled to undergo one forward move, assume it is immediately followed by a $m$-tuple $[t+i]_m$ with $m<i$. Let the predecessor of $[t]_i$ be $[X]_p$, and the successor of $[t+i]_m$ be $[Y]_c$. If we apply one forward move on $[t]_i$, then resulting difference $(t+i)-(t+1)=i-1<i$ would violate the condition \eqref{mat:gap}. So the following adjustment may be performed:
  $$
  \begin{gathered}
 (\text{parts})+[X]_p+[{\bf t+1}]_i+[t+i]_m+[Y]_c+(\text{parts}) \\
  \qquad \downarrow \text {an adjustment on $[t+1]_i$} \\
 (\text{parts})+[X]_p+[t]_m+[{\bf t+m+1}]_i+[Y]_c+(\text{parts}).
  \end{gathered}
  $$
If $p<i$ then $X\leq t-p-1$, and if $p\geq i$ then $X\leq t-p$. This implies that $[X]_p+[t]_m$ is reasonable since $m<i$. On the other hand, observe that if $c\leq m<i$ and $Y=t+i+m$, then the preceding adjustment must be repeated on the $i$-tuple $[t+i+m]_i$ since $(t+i+m)-(t+m+1)=i-1<i$. In all other cases no further adjustment is required. This procedure is iterated until the emerging parts no longer violate the difference condition \eqref{mat:gap}. Therefore, the operation ``adjustment'' is well-defined.
\end{description}
In summary, we have obtained the unique $\la\in \sS^k_{\bn}$. Furthermore, the operation ``forward move'' merely transfers the entire weight of every $\mu^i$ to the base partition $\beta^{\bn}$ without altering either the number of parts $\ell$ or statistics $\ell_i$, while the operation ``adjustment'' leaves the weight, $\ell$ and $\ell_i$ completely unchanged.

Conversely, we may describe the inverse map $\varphi^{-1}$, from any partition $\la\in \sS^k_{\bn}$ back to the $(k+1)$-fold partition vector $(\beta^{\bn}, \mu^1, ..., \mu^k)$. The inverse map processes $i$ from $k$ down to $1$; for each fixed $i$ it applies backward moves to all $i$-tuples of $\la$ in increasing order of their sizes, converting them into $\beta^{(i)}$; the number of backward moves applied to each $i$-tuple, multiplied by $i$, produces all parts of $\mu^i$. As with the forward moves, however, these backward moves may encounter obstructions. We therefore proceed to give detailed definitions of the \emph{backward moves} and \emph{normalizations}.
\begin{description}
\item[Backward moves] For any $i$-tuple $[t]_i$, one backward move transforms it into $[t-1]_i$, decreasing its weight by $i$.

\item[Normalizations] To parallel the treatment of the forward moves, we consider an $i$-tuple $[t+m+1]_i$ that is scheduled to undergo one backward move. Let the predecessor of $[t+m+1]_i$ be $[X]_p+[t]_m$, and the successor of $[t+m+1]_i$ be $[Y]_c$.  
%under the assumption that it is immediately preceded by $s$ consecutive $m$-tuple $[t]_m+[t+2m]_m+\cdots +[t+(s-1)m]_m$ with $m<i$. Let the tuple immediately preceding $[t]_m$ be $[Y]_p$. 
If we apply one backward move on $[t+m+1]_i$, then resulting difference $(t+m)-t=m<m+1$ would violate the condition \eqref{mat:gap}. Hence, the following normalization may be performed:
 $$
  \begin{gathered}
   (\text{parts})+[X]_p+[t]_m+[{\bf t+m}]_i+[Y]_c+(\text{parts})\\
  \qquad \downarrow \text {a normalization on $[t+m]_i$} \\
 (\text{parts})+[X]_p+[{\bf t}]_i+[t+i]_m+[Y]_c+(\text{parts}) .
  \end{gathered}
  $$
In accordance with the order prescribed by the inverse map, we have $c\leq i$, whence $Y\geq t+m+i+1$. This guarantees that $[t+i]_m+[Y]_c$ is reasonable. On the other hand, if $m\leq p<i$ and $X=t-p$, then the preceding normalization must be repeated on the $i$-tuple $[t]_i$ since $t-(t-p)=p<p+1$. In all other cases no further normalizations is required. This procedure is iterated until the emerging parts no longer violate the difference condition \eqref{mat:gap}. Therefore, the operation ``normalization'' is well-defined.
\end{description}
It is clear to see that this is the inverse map of $\varphi$ step by step. In fact, $\varphi$ and $\varphi^{-1}$ proceed in completely opposite orders: $\varphi$ adds the appropriate weights to each $i$-tuple for $i$ increasing, whereas $\varphi^{-1}$ strips off the corresponding weights from each $i$-tuple for $i$ decreasing. Moreover, the forward move and backward move defined above form a pair of inverse operations, and the procedures of ``adjustments'' and ``normalizations'' are also inverse to each other. Then by an argument similar to that for $\varphi$, one can see that $\varphi^{-1}$ also preserves the weight and every statistic. Therefore, we conclude that $\varphi$ is a bijection of the sort described in this lemma.
\end{proof}

At the end of this section, we give an example to illustrate the map $\varphi$; the interested reader is invited to work out its inverse. For ease of presentation, we will consider partition in the set $\sS^3$.

\begin{example}
Given $\beta^{(2, 2, 1)}=[1]_3+[4]_2+[6]_2+[8]_1+[9]_1$ and $(\mu^1, \mu^2, \mu^3)=(1+2, 2+4, 6)$. First apply $\mu^1$ to all $1$-tuples, that is,
\begin{align*}
  & [1]_3+[4]_2+[6]_2+[{\bf 8}]_1+[{\bf 9}]_1\\
%  \qquad \downarrow \text {two move forwards on $[9]_1$ one move forward on $[8]_1$} \\
  \xrightarrow[\text{$\mu^1_1$ forward move on $[8]_1$}]{\text{$\mu^1_2$ forward moves on $[9]_1$}}
 & [1]_3+[4]_2+[6]_2+[{\bf 9}]_1+[{\bf 11}]_1.
\end{align*}
Next, we apply $\mu^2$ to all $2$-tuples with an adjustment, that is,
\begin{align*}
& [1]_3+[4]_2+[{\bf 6}]_2+[9]_1+[11]_1\\
\xrightarrow{\text{$\mu^2_2/2$ forward moves on $[6]_2$}}& [1]_3+[4]_2+[{\bf8}]_2+[9]_1+[11]_1\\
\xrightarrow{\text{an adjustment on $[8]_2$}}& [1]_3+[4]_2+[7]_1+[{\bf9}]_2+[11]_1\\
\xrightarrow{\text{$\mu^2_1/2$ forward move on $[4]_2$}}& [1]_3+[{\bf5}]_2+[7]_1+[9]_2+[11]_1.
\end{align*}
Finally, we apply $\mu^3$ to the $3$-tuple $[1]_3$ with two adjustments, that is,
\begin{align*}
& [{\bf 1}]_3+[5]_2+[7]_1+[9]_2+[11]_1\\
\xrightarrow{\text{$\mu^3_1/3$ forward moves on $[1]_3$}} & [{\bf 3}]_3+[5]_2+[7]_1+[9]_2+[11]_1\\
\xrightarrow{\text{an adjustment on $[3]_3$}} &[2]_2+[{\bf 5}]_3+[7]_1+[9]_2+[11]_1\\
\xrightarrow{\text{an adjustment on $[5]_3$}} &[2]_2+[4]_1+[{\bf 6}]_3+[9]_2+[11]_1.
\end{align*} 
Hence, we obtain $\la=2+2+4+6+6+6+9+9+11\in \sS^3_{(2, 2, 1)}$. Furthermore, 
\begin{align*}
&|\la|=|\beta^{(2, 2, 1)}|+|\mu^1|+|\mu^2|+|\mu^3|=55,\quad \ell(\la)=\ell(\beta^{(2, 2, 1)})=9,\\
&\ell_2(\la)=\ell_2(\beta^{(2, 2, 1)})=2,\quad \ell_3(\la)=\ell_3(\beta^{(2, 2, 1)})=1.
\end{align*}
This yields that the weight and the statistics are preserved.
\end{example}

\section{Conclusion}\label{sec:con}

In the first part of this paper, we investigate two partition sets subject to difference conditions, embed them into the framework of span one linked partition ideals, and obtain two refinements of their generating functions. Going further, a natural question is whether other partition sets can be found that likewise fit this framework; conversely, for a given $q$-series, one may ask whether a partition interpretation can be discovered within this framework. Either direction might facilitate the study of partition identities.

On the other hand, as with the Andrew-Gordon identities in Theorem \ref{thm:AGid}, it is a worthwhile problem to seek a modular form for the multi-sum $q$-series \eqref{id:Ak}, should one exist.

%%%%%%%%%%%%%%%%%%%%%%%%%%%%%%%%%%%%%%
\section*{Acknowledgments}

The author would like to thank Professor Shishuo Fu for helpful discussions. The author is grateful to the anonymous referees for their meticulous suggestions on revising this paper.

%%%%%%%%%%%%%%%%%%%%%%%%%%%%%%%%%%%%%%


\begin{thebibliography}{99}

%\bibitem{and72} G.~E.~Andrews, Two theorems of Gauss and allied identities proved arithmetically, Pacific J. Math. {\bf 41} (1972), 563--578.

%\bibitem{and79} G.~E.~Andrews, Partitions and Durfee dissection, Amer. J. Math. {\bf 101} (1979), 735--742.

\bibitem{and741} G. E. Andrews, A general theory of identities of the Rogers-Ramanujan type, Bull. Am. Math. Soc. {\bf 80} (1974), 1033--1052.

\bibitem{and742} G. E. Andrews, An analytic generalization of the Rogers-Ramanujan identities for odd moduli, Proc. Natl. Acad. Sci. USA {\bf 71} (1974), 4082--4085.

\bibitem{and80} G. E. Andrews, Gap-frequency partitions and the Rogers-Selberg identities, Ars Comb. {\bf 9} (1980), 201--210.

\bibitem{andtp} G.~E.~Andrews, {\it The Theory of Partitions}, Cambridge University Press, Cambridge (1998).

%\bibitem{and20} G.~E.~Andrews, Sequences in partitions, double $q$-series, and the mock theta functions $\rho_3(q)$, from Algorithmic Combinatorics-Enumerative Combinatorics, Special Functions and Computer Algebra, pp.~25--46, Springer, Cham (2020).

%\bibitem{AU23} G.~E.~Andrews and A.~K.~Uncu, {\it Sequences in overpartitions}, The Ramanujan Journal {\bf 61} (2023), 715-729.

%\bibitem{CW23} Z.~Cao and L.~Wang, {\it Multi-sum Rogers-Ramanujan type identities}, Journal of Mathematical Analysis and Applications {\bf 522} (2023) 126960.

%\bibitem{FL24} S.~Fu and H.~Li, Sequences of odd length in strict partitions II: the $2$-measure and refinements of Euler's theorem, preprint.

\bibitem{AC23} G. E. Andrews and S. Chern, Linked partition ideals and a family of quadruple summations, J. Comb. Theory, Ser. A {\bf 200} (2023), 105789.

\bibitem{ACL22} G. E. Andrews, S. Chern and Z. Li, Linked partition ideals and the Alladi-Schur theorem, J. Comb. Theory, Ser. A {\bf 189} (2022), 105614.

\bibitem{che20} S. Chern, Linked partition ideals, directed graphs and $q$-multi-summation, Electron. J. Comb. {\bf 27} (2020), 3.33.

\bibitem{che221} S. Chern, Linked partition ideals and Andrews-Gordon type series for Alladi and Gordon's extension of Schur's identity, Rocky Mt. J. Math. {\bf 52} (2022), 2009--2016.

\bibitem{che222} S. Chern, Linked partition ideals and a Schur-type identity of Andrews, in: Combinatorial and Additive Number Theory V, in: Springer Proc. Math. Stat., vol. 395, Springer, Cham, 2022, pp. 107--117. 

\bibitem{che23} S. Chern, Linked partition ideals and Euclidean billiard partitions, Rev. R. Acad. Cienc. Exactas F\'{i}s. Nat., Ser. A Mat. {\bf 117 (134)} (2023).

\bibitem{CL20} S. Chern and Z. Li, Linked partition ideals and Kanade-Russell conjectures, Discrete Math. {\bf 343} (2020), 111876.

\bibitem{FL241} S.~Fu and H.~Li, Sequences of odd length in strict partitions I: the combinatorics of double sum Rogers-Ramanujan type identities, J. Combin. Theory Ser. A {\bf 219} (2026), Paper No. 106128, 30.

\bibitem{GR04} G.~Gasper and M.~Rahman, {\it Basic Hypergeometric Series, 2nd edition}, Cambridge University Press, Cambridge (2004).

\bibitem{gor61} B. Gordon, A combinatorial generalization of the Rogers-Ramanujan identities, Amer. J. Math. {\bf 83} (1961), 393--399.

\bibitem{GY251} S. Gu and K. Yu, Linked partition ideals and overpartitions, Discrete Math. {\bf 348} (2025), 114380.

\bibitem{GY252} S. Gu and K. Yu, Linked partition ideals and two-color partitions, Ramanujan J. {\bf 67}, 65 (2025).

\bibitem{kil25} Y. C. K\i l\i\c{c}, A combinatorial interpretation of the series for Rogers-Ramanujan-Gordon identities when $k=3$, Ann. Comb. (2025).

%\bibitem{har37} G.~H.~Hardy, {\it Lectures by Godfrey H. Hardy on the mathematical work of Ramanujan}, Edwards Brothers, Ann Arbor, Michigan, 1937, Fall Term 1936. Notes taken by Marshall Hall at the Institute For Advanced Study, Princeton, NJ.

%\bibitem{HW08} G.~H.~Hardy and E.~M.~Wright, {\it An introduction to the theory of numbers, 6th edition}, Oxford University Press, Oxford, 2008.

\bibitem{kur10} K.~Kur\c{s}ung\"{o}z, Parity considerations in Andrews-Gordon identities, European J. Comb. {\bf 31} (2010), 976--1000.

\bibitem{kur19} K.~Kur\c{s}ung\"{o}z, Andrews-Gordon type series for Capparelli's and G\"{o}llnitz-Gordon identities, J. Comb. Theory Ser. A {\bf 165} (2019), 117--138.

\bibitem{kur21} K.~Kur\c{s}ung\"{o}z, Andrews-Gordon type series for Schur's partition identity, Discrete Math. {\bf 344} (2021), 112563.

\bibitem{KS22} K.~Kur\c{s}ung\"{o}z and H.~\"O.~Seyrek, Construction of evidently positive series and an alternative construction for a family of partition generating functions due to Kanade and Russell, Ann. Comb. {\bf 26} (2022), 903--942.

\bibitem{mac16} P. A. MacMahon, {\it Combinatory Analysis vol. 2}, Cambridge University Press, New York, NY, USA, 1916.

\bibitem{ram14} S. Ramanujan, Problem 584, J. Indian Math. Soc. {\bf 6} (1914), 199--200.

\bibitem{ram19}S. Ramanujan, Proof of certain identities in combinatory analysis, Proc. Camb, Philos. Soc. {\bf 19} (1919), 214--216.

\bibitem{rog94} L. J. Rogers, Second memoir on the expansion of certain infinite products, Proc. Lond. Math. Soc. {\bf 25} (1894), 318--343.

\bibitem{sch17} I. Schur, Ein Beitrag zur Additiven Zahlentheorie und zur Theorie der Kettenbr\"{u}che, Sitzungsber. Preuss. Akad. Wiss. Phys.-Math. Klasse (1917), 302--321.

%\bibitem{lov06} J.~Lovejoy, {\it Constant terms, jagged partitions, and partitions with difference two at distance two}, Aequ. Math. {\bf 72} (2006), 299-312.

%\bibitem{LW23} Z.~Li and L.~Wang, {\it Rogers-Ramanujan type identities involving double, triple and quadruple sums}, \href{https://arxiv.org/abs/2306.17085v1}{arXiv:2306.17085v1}.

%\bibitem{mac15} P.~A.~MacMahon, {\it Combinatory Analysis}, Cambridge University Press, London (1915/16).

% \bibitem{mac84} P.~A.~MacMahon, {\it Combinatory Analysis, vol. 2.} A.M.S. Chelsea Publishing, Providence (1984).

%\bibitem{sil18} A.~V.~Sills, {\it An invitation to the Rogers-Ramanujan identities.} CRC Press, Boca Raton (2018).

%\bibitem{sla52} L.~J.~Slater, Further identities of the Rogers-Ramanujan type, Proc. London Math. Soc. (2) {\bf 54} (1952), 147--167.

%\bibitem{syl82} J.~J.~Sylvester, A constructive theory of partitions in three acts, an interact and a exodion, Am. J. Math. {\bf 5} (1882), 251--330.

%\bibitem{WW23} B.~Wang and L.~Wang, {\it Proofs of Mizuno's conjectures on generalized rank two Nahm sums}, \href{https://arxiv.org/abs/2308.14728}{arXiv:2308.14728}.

%\bibitem{wan22} L.~Wang, {\it Identities on Zagier's rank two examples for Nahm's conjecture}, \href{https://arxiv.org/abs/2210.10748v2}{arXiv:2210.10748v2}.

%\bibitem{war97} S.~O.~Warnaar, The Andrews-Gordon identities and $q$-multinomial coefficients, Comm. Math. Phys. {\bf 184} (1997), 203--232.

%\bibitem{WYR23} C.~Wei, Y.~Yu and G.~Ruan, {\it Multidimensional Rogers-Ramanujan type identities with parameters}, \href{https://arxiv.org/abs/2302.00357v3}{arXiv:2302.00357v3}.

% \bibitem{RPS} R.~P.~Stanley, {\it Enumerative Combinatorics, vol. I, 2nd edition}, Cambridge Studies in Advanced Mathematics, Cambridge University Press, (2011).

% \bibitem{all20} K.~Alladi, Euler's partition theorem and refinements without appeal to infinite products, in: {\it Algorithmic Combinatorics: Enumerative Combinatorics, Special Functions and Computer Algebra}, Springer, Cham. Switzerland, 2020, pp.~9--23.

% \bibitem{AB} K.~Alladi and A.~Berkovich, New polynomial analogues of Jacobi's triple product and Lebesgue's identities, Adv. in Appl. Math. {\bf 32} (2004), 801--824.

% \bibitem{AG} K.~Alladi and B.~Gordon, Partition identities and a continued fraction of Ramanujan, J. Combin. Theory Ser. A {\bf 63} (1993), 275--300.

% \bibitem{and87} G.~E.~Andrews, Rogers-Ramanujan identities for two-color partitions, Indian J. Math. {\bf 29} (1987), 117--125.

% \bibitem{and15} G.~E.~Andrews, Basis partition polynomials, overpartitions and the Rogers-Ramanujan identities, J. Approx. Theory {\bf 197} (2015), 62--68.

% \bibitem{and21} G.~E.~Andrews, Partition identities for two-color partitions, Hardy-Ramanujan Journal {\bf 44} (2021), 74--80.

% \bibitem{ABD22}
% G.~E.~Andrews, S.~Bhattacharjee, and M.~G.~Dastidar, Sequences in partitions, Integers \textbf{22} (2022), Paper No.~A32, 9 pp.

% \bibitem{ACL22}
% G.~E.~Andrews, S.~Chern and Z.~Li, On the $k$-measure of partitions and distinct partitions, Algebr. Comb. \textbf{5} (2022), 1353--1361.

% \bibitem{BU} A.~Berkovich and A.~K.~Uncu, Elementary polynomial identities involving $q$-trinomial coefficients, Ann. Comb. {\bf 23} (3-4) (2019), 549--560.

% \bibitem{bes94} C.~Bessenrodt, A bijection for Lebesgue's partition identity in the spirit of Sylvester, Discrete Math. {\bf 132} (1994), 1--10.

% \bibitem{CHS} W.~Y.~C. Chen, Q.-H. Hou, and L.~H. Sun, {\it An iterated map for the Lebesgue identity}, \href{https://arxiv.org/abs/1002.0135}{arXiv:1002.0135v2}.

% \bibitem{CFT} S.~Chern, S.~Fu, and D.~Tang, Multi-dimensional $q$-summations and multi-colored partitions, Ramanujan J. {\bf 51} (2020), 297--306.

% \bibitem{CJ} S.~Corteel and J.~Lovejoy, Overpartitions, Trans. Amer. Math. Soc. {\bf 356} (2004), 1623--1635.

% \bibitem{DK} J.~Dousse and B.~Kim, An overpartition analogue of $q$-binomial coefficients, II: Combinatorial proofs and $(q,t)$-log concavity, J. Combin. Theory Ser. A {\bf 158} (2018), 228--253.

% \bibitem{FT18} S.~Fu and D.~Tang, Partitions with fixed largest hook length, Ramanujan J. {\bf 45} (2018), 375--390.

% \bibitem{gup78} H.~Gupta, The rank-vector of a partition, Fibonacci Quart. {\bf 16} (1978), 548--552.

% \bibitem{KN} W.~J.~Keith and R.~Nath, Partitions with prescribed hooksets, J. Comb. Number Theory {\bf 3}(1) (2011), 39--50.

% \bibitem{LO09} J.~Lovejoy and R.~Osburn, $M_2$-rank differences for partitions without repeated odd parts, J. Th\'eor. Nombres Bordeaux {\bf21} (2009): 313--334.

% \bibitem{mac23} P.~A.~MacMahon, The theory of modular partitions, Proc. Cambridge Phil. Soc. {\bf 21} (1923), 197--204.

% \bibitem{pak06} I.~Pak, Partition bijections, a survey, Ramanujan J. {\bf 12} (2006), 5--75.

% \bibitem{unc21} A.~K.~Uncu, On double sum generating functions in connection with some classical partition theorems, Discrete Math. {\bf 344} (2021) 112562.

% \bibitem{RPS} R.~P.~Stanley, {\it Enumerative Combinatorics, vol. I, 2nd edition}, Cambridge Studies in Advanced Mathematics, Cambridge University Press, (2011).

% \bibitem{zen05} J.~Zeng, The $q$-variations of Sylvester's bijection between odd and strict partitions, Ramanujan J. {\bf 9} (2005), 289--303.

\end{thebibliography}
\end{document}